\documentclass{amsart}

\input xy
\xyoption{all}

\newcommand{\pU}{p\mathcal U}
\newcommand{\qU}{q\mathcal U}
\newcommand{\qRU}{\ddot{\mathcal Q}}

\newcommand{\IR}{\mathbb R}
\newcommand{\IZ}{\mathbb Z}
\newcommand{\U}{\mathcal U}
\newcommand{\V}{\mathcal V}
\newcommand{\W}{\mathcal W}
\newcommand{\N}{\mathcal N}

\newcommand{\IQ}{\mathbb Q}
\newcommand{\Ra}{\Rightarrow}

\newcommand{\IN}{\mathbb N}

\newcommand{\w}{\omega}
\newcommand{\FU}{{\mathcal Q}}

\newcommand{\rcluwL}{{}^\circ\kern-2pt\overline{uw}L}
\newcommand{\Hs}{H\kern-2pt s}
\newcommand{\St}{\mathcal{S}t}
\newcommand{\Rpsi}{\overset{...}{\kern-2pt\psi}}
\newcommand{\Tb}{\overset{...}{\kern-2pt b}}
\newcommand{\supp}{\mathrm{supp}}

\newcommand{\F}{\mathcal F}
\newcommand{\A}{\mathcal A}
\newcommand{\dc}{dc}

\newtheorem{theorem}{Theorem}[section]
\newtheorem{corollary}[theorem]{Corollary}
\newtheorem{proposition}[theorem]{Proposition}
\newtheorem{problem}[theorem]{Problem}
\newtheorem{question}[theorem]{Question}
\newtheorem{claim}[theorem]{Claim}
\newtheorem{lemma}[theorem]{Lemma}
\theoremstyle{definition}
\newtheorem{definition}[theorem]{Definition}
\newtheorem{example}[theorem]{Example}

\title[The submetrizablity and $i$-weight of paratopological groups]{On the submetrizability number and $i$-weight\\ of quasi-uniform spaces and paratopological groups}
\author{Taras Banakh and Alex Ravsky}
\address{T.Banakh: Ivan Franko National University of Lviv (Ukraine), and Jan Kochanowski University in Kielce (Poland)}
\email{t.o.banakh@gmail.com}
\address{A.Ravsky: Pidstryhach Institute for Applied Problems of Mechanics and Mathematics of National Academy of Sciences, Lviv, Ukraine}
\email{oravsky@mail.ru}
\thanks{The first author has been partially financed by NCN grant DEC-2012/07/D/ST1/02087.}

\begin{document}
\begin{abstract} We derive many upper bounds on the submetrizability number and $i$-weight of paratopological groups and topological monoids with open shifts. In particular, we prove that each first countable Hausdorff paratopological group is submetrizable thus answering a problem of Arhangelskii posed in 2002. Also we construct an example of a zero-dimensional (and hence regular) Hausdorff paratopological abelian group $G$ with countable pseudocharacter which is not submetrizable. In fact, all results on the $i$-weight and submetrizability are derived from more general results concerning normally quasi-uniformizable and bi-quasi-uniformizable spaces.
\end{abstract}
\maketitle

\section*{Introduction}

This paper was motivated by the following problem of Arhangelskii \cite[3.11]{Ar} (also repeated by Tkachenko in his survey \cite[2.1]{Tka}):
{\em Does every first countable Hausdorff paratopological group admit a weaker metrizable topology?}
A surprisingly simple answer to this problem was given by the authors in \cite{BR}. We just observed that each Hausdorff paratopological group $G$ carries a natural uniformity generated by the base consisting of entourages $\{(x,y)\in G\times G:y\in UxU^{-1}\cap U^{-1}xU\}$ where $U$ runs over open neighborhoods of the unit $e$ in $G$. In \cite{BR} this uniformity was called the {\em quasi-Roelcke uniformity} on $G$ and denoted by $\FU$. If $G$ is first-countable, then the quasi-Roelcke uniformity $\FU$ is metrizable, which implies that the space $G$ is submetrizable.
Moreover, if the quasi-Roelcke uniformity $\FU$ is $\w$-bounded, then the topology generated by the uniformity $\FU$ is metrizable and separable, which implies that $G$ has countable $i$-weight, i.e., admits a continuous injective map onto a metrizable separable space.

In fact, for the submetrizability of $G$ it suffices to require the countability of the pseudocharacter $\psi(\FU)$ of $\FU$, i.e., the existence of a countable subfamily $\U\subset\FU$ such that $\bigcap\U=\Delta_X$. So, the aim of the paper is to detect paratopological groups $G$ whose quasi-Roelcke uniformity $\FU$ has countable pseudocharacter. For this we shall find some upper bounds on the pseudocharacter $\psi(\FU)$. These bounds will give us upper bounds on the submetrizability number $sm(G)$ and the $i$-weight $iw(G)$ of a paratopological group $G$. In fact, the obtained upper bounds on $sm(G)$ and $iw(G)$ have uniform nature and depends on the properties of the two canonical quasi-uniformities $\mathcal L$ and $\mathcal R$ on $G$ called the left and right quasi-uniformities of $G$. These quasi-uniformities are studied in Sections~\ref{s6} and \ref{s7}. In Sections~\ref{s3} and \ref{s4} we study properties of topological spaces whose topology is generated by two quasi-uniformities which are compatible in some sense (more precisely, are $\pm$-subcommuting or normally $\pm$-subcommuting). In Section~\ref{s4} we prove that any two normally $\pm$-subcommuting quasi-uniformities are normal in the sense of \cite{BR}. This motivates the study of topological spaces whose topology is generated by a normal quasi-uniformity. For such spaces we obtain some upper bounds on the $i$-weight, which is done in Section~\ref{s4}. Section~\ref{s1} has preliminary character. It contains the necessary information of topological spaces, quasi-uniform spaces, and their cardinal characteristics. In Section~\ref{s8} we present two counterexamples to some natural conjectures concerning submetrizable paratopological groups.

\section{Preliminaries}\label{s1}

In this section we collect known information on topological spaces, quasi-uniformities, and their cardinal characteristics. For a set $X$ by $|X|$ we denote its cardinality. By $\w$ we denote the set of all finite ordinals and by $\IN=\w\setminus\{0\}$ the set of natural numbers.

For a cardinal $\kappa$ by $\log(\kappa)$ we denote the smallest cardinal $\lambda$ such that $2^\lambda\ge \kappa$.

\subsection{Topological spaces and their cardinal characteristics}\label{s11}

For a subset $A$ of a topological space $X$ by $\overline{A}$ and $A^\circ$ and $\overline{A}^\circ$ we  denote the closure, interior and interior of the closure of the set $A$ in $X$, respectively.

A family $\mathcal N$ of subsets of a topological space $X$ is called a {\em network} of the topology of $X$ if each open set $U\subset X$ can be written as the union $\bigcup\U$ of some subfamily $\U\subset \mathcal N$. If each set $N\in\mathcal N$ is open in $X$, then $\mathcal N$ is a {\em base} of the topology of $X$.

A subset $D$ of a topological space $X$ is called {\em strongly discrete} if each point $x\in D$ has a neighborhood $U_x\subset X$ such that the family $(U_x)_{x\in D}$ is {\em discrete} in the sense that each point $z\in X$ has a neighborhood that meets at most one set $U_x$, $x\in D$. It is easy to see that each strongly discrete subset of (a $T_1$-space) $X$ is discrete (and closed) in $X$. A topological space $X$ is called ({\em strongly}) {\em $\sigma$-discrete} if $X$ can be written as the countable union $X=\bigcup_{n\in\w}X_n$ of (strongly) discrete subsets of $X$.
\smallskip

A topological space $X$ is called
\begin{itemize}
\item {\em Hausdorff\/} if any two distinct points $x,y\in X$ have disjoint open neighborhoods $O_x\ni x$ and $O_y\ni y$;
\item {\em collectively Hausdorff} if each closed discrete subset of $X$ is strongly discrete in $X$;
\item {\em functionally Hausdorff\/} if for any two distinct points $x,y\in X$ there is a continuous function $f:X\to\IR$ such that $f(x)\ne f(y)$;
\item {\em regular} if for any  point $x\in X$ and a neighborhood $O_x\subset X$ there is a neighborhood $V_x\subset X$ of $x$ such that $\overline{V}_x\subset O_x$;
\item {\em completely regular} if for any  point $x\in X$ and a neighborhood $O_x\subset X$ there is a continuous function $f:X\to [0,1]$ such that $f(x)=0$ and $f^{-1}\big([0,1)\big)\subset O_x$;
\item {\em quasi-regular} if each non-empty open set $U\subset X$ contains the closure $\overline{V}$ of another non-empty open set $V\subset X$;
\item {\em submetrizable} if $X$ admits a continuous metric (or equivalently, admits a continuous injective map into a metrizable space).
\end{itemize}
It is clear that each submetrizable space is functionally Hausdorff.

In Section~\ref{s8} will shall need the following property of strongly $\sigma$-discrete spaces.

\begin{proposition}\label{p1.1} Each strongly $\sigma$-discrete Tychonoff space $X$ is zero-dimensional and submetrizable. Moreover, $X$ admits an injective continuous map into the Cantor cube $\{0,1\}^\kappa$ of weight $\kappa=\log(|X|)$.
\end{proposition}

\begin{proof} The proposition trivially holds if $X$ is discrete. So, we assume that $X$ is not discrete and hence infinite. Write $X$ as the countable union $X=\bigcup_{n\in\w}X_n$ of pairwise disjoint strongly discrete non-empty subsets $X_n$ of $X$. Let $\beta X$ be the Stone-\v Cech compactification of $X$. Using the strong discreteness of each $X_n$, we can extend each continuous bounded function $f:X_n\to \IR$ to a continuous bounded function on $X$. This implies that the closure $\bar X_n$ of $X_n$ in $\beta X$ is homeomorphic to the Stone-Cech compactification $\beta X_n$ of the discrete space $X_n$ and hence has covering dimension $\dim(\beta X_n)=0$ (see \cite[3.6.7 and 7.1.17]{Eng}). By the Countable Sum Theorem \cite[3.1.8]{End} for covering dimension in normal spaces, the $\sigma$-compact (and hence normal) space $Z=\bigcup_{n\in\w}\overline{X}_n$ has covering dimension $\dim(Z)=0$, which implies that its subspace $X=\bigcup_{n\in\w}X_n$ is zero-dimensional.

Now we prove that $X$ is submetrizable. For every $n\in\w$ and every $x\in X_n$ we can choose a closed-and-open neighborhood $U_x\subset X$ of $x$ such that $U_x\cap \bigcup_{k<n}X_k=\emptyset$ and the indexed family $(U_x)_{x\in X_n}$ is discrete in $X$. Then the union $\bigcup_{x\in X_n}U_x$ is a closed-and-open subset in $X$ and the function $d_n:X\times X\to\{0,1\}$ defined by
$$d_n(x,y)=\begin{cases}
0,&\mbox{if $x,y\in U_x$ for some $x\in X_n$ or $x,y\notin\bigcup_{z\in X_n}U_z$},\\
1,&\mbox{otherwise,}
\end{cases}$$
is a continuous pseudometric on $X$.
Consequently, the function $d=\max_{n\in\w}\frac1{2^n}d_n$ is a continuous metric on $X$, which implies that $X$ is submetrizable.

It follows that the space $X$ admits a continuous injective map into the countable product $\prod_{n\in\w}D_n$ of discrete spaces $D_n$ of cardinality $|D_n|=1+|X_n|\le |X|$. By definition of the cardinal $\kappa=\log(|X|)$, every discrete space $D_n$, $n\in\w$, admits an injective (and necessarily continuous) map into the Cantor cube $\{0,1\}^\kappa$. Then $\prod_{n\in\w}D_n$ and hence $X$ also admits a continuous injective map into $\{0,1\}^\kappa$.
\end{proof}

For a cover $\U$ of a set $X$ and a subset $A\subset X$ we put $\St^0(A;\U)=A$ and $\St^{n+1}(A;\U)=\bigcup\{U\in\U:U\cap\St^n(A;\U)\ne\emptyset\}$ for $n\ge 0$.

%Let $\mathcal P$ be a property of topological spaces. Following \cite{vMTW}, we shall say that a topological space $X$ is
%\begin{itemize}
%\item {\em star $\mathcal P$} of for any open cover $\U$ there exist a subspace $Y\subset X$ with property $\mathcal P$ such that $X=\St(Y;\U)$;
%\item {\em dually $\mathcal P$} if for any family $\{U_x\}_{x\in X}$ of open sets $U_x\ni x$ there is a subspace $Y\subset X$ with property $\mathcal P$ such that $X=\bigcup_{y\in Y}U_y$.
%\end{itemize}

\subsection{Cardinal characteristics of topological spaces, I}\label{s12}
For a topological space $X$ let
\begin{itemize}
\item $nw(X)=\min\{|\mathcal N|:\mathcal N$ is a network of the topology of $X\}$ be the {\em network weight} of $X$;
\item $d(X)=\min\{|A|:A\subset X,\;\overline{A}=X\}$ be the {\em density} of $X$;
\item $hd(X)=\sup\{d(Y):Y\subset X\}$ the {\em hereditary density} of $X$;
\item $s(X)=\sup\{|D|:D$ is a discrete subspace of $X\}$ be the {\em spread} of $X$;
\item $e(X)=\sup\{|D|:D$ is a closed discrete subspace of $X\}$ be the {\em extent} of $X$;
\item $c(X)=\sup\{|\U|:\U$ is a disjoint family of non-empty open sets in $X\}$ be the {\em cellularity} of $X$;
\item $de(X)=\sup\{|\U|:\U$ is a discrete family of non-empty sets in $X\}$ be the {\em discrete extent} of $X$;
\item $\dc(X)=\sup\{|\U|:\U$ is a discrete family of non-empty open sets in $X\}$ be the {\em discrete cellularity} of $X$;
    \smallskip
\item $l(X)$, the {\em Lindel\"of number} of $X$, be the smallest cardinal $\kappa$ such that each open cover $\U$ of $X$ has a subcover $\V\subset\U$ of cardinality $|\V|\le\kappa$;
\item $\bar l(X)$, the {\em weak Lindel\"of number} of $X$, be the smallest cardinal $\kappa$ such that each open cover $\U$ of $X$ contains a subcollection $\V\subset\U$ of cardinality $|\V|\le\kappa$ with dense union $\bigcup\V$ in $X$;
\item $l^*(X)$, the {\em weak extent} of $X$, be the smallest cardinal $\kappa$ such that for each open cover $\U$ of $X$ there is a subset $A\subset X$ of cardinality $|A|\le\kappa$ such that $X=\St(A;\U)$.
\end{itemize}
The cardinal characteristics $nw,d,s,e,c,l$ are well-known in General Topology (see \cite{Eng}, \cite{Hod}) whereas $\bar l$, $\bar l^*$ are relatively new and notations for these cardinal characteristics are not fixed yet. For example, the weak Lindel\"of number $\bar l$ often is denoted by $wL$, but in \cite[\S3]{Hod} it is denoted by $wc$ and called the {\em weak covering number}.
According to \cite{vMTW}, the weak extent $l^*$ can be called the star cardinality. Spaces with countable weak extent are called star-Lindel\"of in \cite{Mat} and strongly star-Lindel\"of in \cite{DRRT}. Observe that $e\le de$ and $e(X)=de(X)$ for any $T_1$-space $X$.

The relations between the above cardinal invariants are described in the following version of Hodel's diagram \cite{Hod}. In this diagram an arrow $f\to g$ (resp $f\dashrightarrow g$)
indicates that $f(X)\le g(X)$ for any ($T_1$-) space $X$.
$$
\xymatrix{
l^*\ar[r]&de\ar[r]\ar[rd]\ar@/_/@{-->}[d]&l\ar[r]&hl\ar[rd]\\
\dc\ar[ru]\ar[rd]&e\ar[u]&s\ar[ru]\ar[rd]&&nw\ar[r]&w\\
\bar l\ar[r]&c\ar[ru]\ar[r]&d\ar[r]&hd\ar[ru]
}$$

In fact, the cardinal characteristics $d$, $l$, $\bar l$, $l^*$ are initial representatives of the hierarchy of cardinal characteristics $l^{*n}$ and $\bar l^{*n}$, $n\in\frac12\IN$, describing star-covering properties of topological spaces (see the survey paper \cite{Mat} of Matveev for more information on this subject).

For a topological space $X$ and an integer number $n\ge 0$ let
\begin{itemize}
\item $l^{*n}(X)$ be the smallest cardinal $\kappa$ such that for every open cover $\U$ of $X$ there is a subset $A\subset X$ of cardinality $|A|\le\kappa$ such that  $\St^n(A;\U)=X$;
\item $\bar l^{*n}(X)$ be the smallest cardinal $\kappa$ such that for every open cover $\U$ of $X$ there is a subset $A\subset X$ of cardinality $|A|\le\kappa$ such that  $\St^n(A;\U)$ is dense in $X$;
\item $l^{*n\kern-1pt\frac12}(X)$ be the smallest cardinal $\kappa$ such that every open cover $\U$ of $X$ contains a subfamily $\V\subset\U$ of cardinality $|\V|\le\kappa$ such that $\St^n(\cup\V;\U)=X$;
\item $\bar l^{*n\kern-1pt\frac12}(X)$ be the smallest cardinal $\kappa$ such that every open cover $\U$ of $X$ contains a subfamily $\V\subset\U$ of cardinality $|\V|\le\kappa$ such that $\St^n(\cup\V;\U)$ is dense in $X$;
\item $l^{*\w}(X)=\min_{n\in\w}l^{*n}(X)$ and $\bar l^{*\w}=\min_{n\in\w}\bar l^{*n}(X)$.
\end{itemize}
Observe that $l^{*0}=|\cdot|$, $\bar l^{*0}=d$, $l^{*\frac12}=l$, $\bar l^{*\frac12}=\bar l$, and
$l^{*1}=l^*$.

In \cite{Cao} the cardinal characteristics $l^{*n}$ and $l^{*n\kern-1pt\frac12}$ are denoted by $st_{n}\mbox{-}l$ and $st_{n\frac12}\mbox{-}l$, respectively. In \cite{DRRT} spaces $X$ with countable $l^{*n\kern-1pt\frac12}(X)$ and $l^{*n}(X)$ are called {\em $n$-star-Lindel\"of} and {\em strongly $n$-star Lindel\"of}, respectively.

The following diagram describes provable inequalities between cardinal characteristics $l^{*n}$, $\bar l^{*n}$, $l^{*n\kern-1pt \frac12}$, and $\bar l^{*n\kern-1pt\frac12}$ for $n\in\IN$. For two cardinal characteristics $f,g$ an arrow $f\to g$ indicates that $f(X)\le g(X)$ for any topological space $X$.
$$\xymatrix{
l^{*\w}\ar@{=}[dd]\ar[r]&\cdots\ar[r]&l^{*(n{+}1)}\ar[r]\ar[rdd]&l^{*(n{+}\frac12)}\ar[r]\ar[rdd]&l^{*n}\ar[r]
&\cdots\ar[r]&l^{*1}=l^*\ar[r]\ar[rrdd]&de\ar[r]\ar[rd]&l^{*\frac12}=l\ar[r]&hl\ar[d]\\
&&&&&&dc\ar[ru]\ar[rd]&&s\ar[r]\ar[ru]\ar[rd]&nw\\
\bar l^{*\w}\ar[r]&\cdots\ar[r]&\bar l^{*(n{+}\frac12)}\ar[r]\ar[ruu]&\bar l^{*n}\ar[r]\ar[ruu]&\bar l^{*(n{-}\frac12)}\ar[r]\ar[rru]&\cdots\ar[r]&\bar l^{*\frac12}=\bar l\ar[r]\ar[rruu]&c\ar[r]\ar[ru]&\bar l^{*0}=d\ar[r]&hd\ar[u]
}
$$
\smallskip

The unique non-trivial inequalities $l^{*1}\le de$ and $\bar l^{*1\kern-1pt\frac12}\le dc$ in this diagram follow from the next proposition whose proof can be found in \cite{BRv}.

\begin{proposition}\label{p1.2a} Any topological space $X$ has $l^{*1}(X)\le de(X)$ and $\bar l^{*1\kern-1pt\frac12}(X)\le \dc(X)$.
\end{proposition}

For quasi-regular spaces many star-covering properties are equivalent. Let us recall that a topological space $X$ is called {\em quasi-regular} if each non-empty open set $U\subset X$ contains the closure $\overline{V}$ of another non-empty open set $V$ in $X$. The following proposition was proved in \cite{BRv} (and for regular spaces in \cite{DRRT}).

\begin{proposition}\label{p1.2} Let $X$ be a quasi-regular space. Then
\begin{enumerate}
\item $\dc(X)=\bar l^{*1\kern-1pt\frac12}(X)=l^{*\w}(X)$.
\item If $X$ is normal, then $\dc(X)=\bar l^{*1}(X)$.
\item If $X$ is perfectly normal, then $\dc(X)=c(X)=\bar l^{*\frac12}(X)$.
\item If $X$ is collectively Hausdorff, then $\dc(X)=de(X)=l^{*\kern-1pt 1}(X)$.
\item If $X$ is paracompact, then $\dc(X)=l(X)$.
\item If $X$ is perfectly paracompact, then $\dc(X)=hl(X)$.
\end{enumerate}
\end{proposition}

Proposition~\ref{p1.2} implies that for quasi-regular spaces the diagram describing the relations between the cardinal characteristics simplifies to the following form.
$$
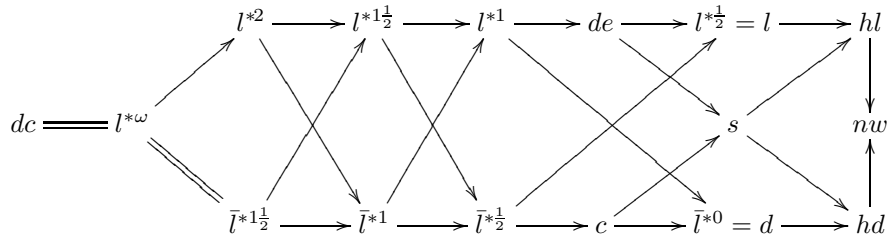
\xymatrix{
%&&&&e\ar[rd]\\
&&l^{*\kern-1pt 2}\ar[r]\ar[rdd]&l^{*1\kern-1pt\frac12}\ar[r]\ar[rdd]&l^{*\kern-0.5pt 1}\ar[r]\ar[rrdd]&de\ar[r]\ar[rd]&l^{*\kern-1pt\frac12}=l\ar[r]&hl\ar[d]\\
\dc\ar@{=}[r]&l^{*\w}\ar@{=}[rd]\ar[ru]&&&&&s\ar[ru]\ar[rd]&nw\\
&&\bar l^{*1\kern-1pt\frac12}\ar[r]\ar[ruu]&\bar l^{*1}\ar[r]\ar[ruu]&\bar l^{*\kern-1pt\frac12}\ar[r]\ar[rruu]&c\ar[r]\ar[ru]&\bar l^{*0}=d\ar[r]&hd\ar[u]
}
$$

Next, we consider some local cardinal characteristics of topological spaces.
Let $X$ be a topological space, $x$ be a point of $X$, and $\mathcal N_x$ be the family of all open neighborhoods of $x$ in $X$.
\begin{itemize}
\item The {\em character} $\chi_x(X)$ of $X$ at $x$ is the smallest cardinality of a neighborhood base at $x$.
\item The {\em pseudocharacter} $\psi_x(X)$ of $X$ at $x$ is the smallest cardinality of a subfamily $\U\subset\mathcal N_x$ such that $\bigcap\U=\bigcap\N_x$.
\item The {\em closed pseudocharacter} $\overline{\psi}_x(X)$ of $X$ at $x$ is the smallest cardinality of a subfamily $\U\subset\mathcal N_x$ such that $\bigcap_{U\in\U}\overline{U}=\bigcap_{V\in\N_x}\overline{V}$.
\end{itemize}
It is easy to see that for any point $x$ of a Hausdorff topological space $X$ we get
$$\psi_x(X)\le\overline{\psi}_x(X)\le\chi_x(X).$$
The cardinals $$\chi(X)=\sup_{x\in X}\chi_x(X),\;\;\psi(X)=\sup_{x\in X}\psi_x(X),\;\;\mbox{ and }\;\;\overline{\psi}(X)=\sup_{x\in X}\overline{\psi}_x(X)$$
are called the {\em character}, the {\em pseudocharacter}, and the {\em closed pseudocharacter} of $X$, respectively.
It follows that $$\psi(X)\le\overline{\psi}(X)\le\chi(X)$$for any Hausdorff topological space $X$.

The (closed) pseudocharacter is upper bounded by the (closed) diagonal number defined as follows.
Let $X$ be a Hausdorff topological space. By $\Delta_X=\{(x,y)\in X\times X:x=y\}$ we denote the {\em diagonal} of the square $X\times X$.
\begin{itemize}
\item The {\em diagonal number} $\Delta(X)$ of $X$ is the smallest cardinality of a family $\U$ of open subsets of $X\times X$ such that $\bigcap\U=\Delta_X$.
\item The {\em closed diagonal number} $\overline{\Delta}(X)$ of $X$ is the smallest cardinality of a family $\U$ of open subsets of $X\times X$ such that $\bigcap_{U\in\U}\overline{U}=\Delta_X$.
\end{itemize}
It is easy to see that $\psi(X)\le\Delta(X)\le\overline{\Delta}(X)$ and $\overline{\psi}(X)\le\overline{\Delta}(X)$ for any Hausdorff space $X$.

Following \cite[\S2.1]{Grue} we say that a space $X$ has ({\em regular}) {\em $G_\delta$-diagonal} if $\Delta(X)\le\w$ (resp. $\overline{\Delta}(X)\le\w$).
\smallskip

The (closed) diagonal number of a functionally Hausdorff space $X$ is upper bounded by
\begin{itemize}
\item the {\em submetrizability number} $sm(X)$ of $X$, defined as the smallest number of continuous pseudometrics which separate points of $X$, and
\item the {\em $i$-weight} $iw(X)$ of $X$, defined as the smallest number of continuous real-valued functions that separate points of $X$.
\end{itemize}

The following diagram describes relations between these cardinal characteristics. In this diagram for two cardinal characteristics $f,g$ an arrow $f\to g$ indicates that $f(X)\le g(X)$ for any functionally Hausdorff topological space $X$.
$$\xymatrix{
&\psi\ar[r]\ar[d]&\Delta\ar[d]\\
\chi&\overline{\psi}\ar[l]\ar[r]&\overline{\Delta}\ar[r]&sm\ar[r]&iw\ar[r]&sm\cdot \log\dc
}
$$

{The unique non-trivial inequality $iw\le sm\cdot \log\dc$ in this diagram is proved in the following proposition.

\begin{proposition}\label{p1.3} Each infinite functionally Hausdorff space $X$ has  $$iw(X)\cdot \w=sm(X)\cdot \log(\dc(X))\mbox{ \ \ and \ \ }|X|\le\dc(X)^{\w\cdot sm(X)}\le 2^{\w\cdot iw(X)}.$$
\end{proposition}

\begin{proof} The inequality $sm(X)\cdot\log(\dc(X))\le iw(X)\cdot\w$ follows from the inequalities $sm(X)\le iw(X)$ and $\dc(X)\le |X|\le |[0,1]^{iw(X)}|=2^{iw(X)\cdot\w}$, the latter of which implies $\log(\dc(X))\le \log(2^{iw(X)\cdot\w})\le iw(X)\cdot\w$.

Now we prove the inequalities $iw(X)\cdot \w\le sm(X)\cdot \log(\dc(X))$ and $|X|\le dc(X)^{\w\cdot sm(X)}$. The definition of the submetrizability number implies that $X$ admits a continuous injective map $f:X\to \prod_{\alpha\in sm(X)}M_\alpha$ into the Tychonoff product of $sm(X)$ many metric spaces $M_\alpha$. We lose no generality assuming that each metric space $M_\alpha$ is a continuous image of $X$ and hence $d(M_\alpha)=\dc(M_\alpha)\le \dc(X)$ and $|M_\alpha|\le d(M_\alpha)^\w$. Then $$|X|\le\prod_{\alpha\in sm(X)}|M_\alpha|\le \prod_{\alpha\in sm(X)}d(M_\alpha)^\w\le \prod_{\alpha\in sm(X)}dc(X)^\w=dc(X)^{\w\cdot sm(X)}.$$

By \cite[4.4.9]{Eng}, for every $\alpha\in sm(X)$ the metric space $M_\alpha$ admits a topological embedding into the countable power $H_{\kappa}^\w$ of the hedgehog $H_{\kappa}=\{(x_i)_{i\in\kappa}\in[0,1]^\kappa:|\{i\in\kappa:x_i\ne0\}|\le 1\}$ with $\kappa=\dc(X)\ge d(M_\alpha)$ many spines. The hedgehog $H_\kappa$ can be thought as a cone over a discrete space $D$ of cardinality $\kappa$. The discrete space $D$ admits an injective continuous map into the Tychonoff cube $[0,1]^{\log(\kappa)}$. Consequently, $H_\kappa$ admits an injective continuous map into the cone over the Tychonoff cube $[0,1]^{\log(\kappa)}$, which implies that $iw(H_k)\le \log(\kappa)=\log(\dc(X))$ and $iw(H_\kappa^\w)\le \log(\dc(X))\cdot\w=\log(\dc(X))$. Then $iw(X)\le sm(X)\cdot iw(H_k^\w)\le sm(X)\cdot \log(\dc(X))$. This completes the proof of the equality $iw(X)\cdot \w=sm(X)\cdot\log(\dc(X))$.

To complete the proof of the proposition, observe that
$$|X|\le\dc(X)^{\w\cdot sm(X)}\le \big(2^{\log\dc(X)}\big)^{\w\cdot sm(X)}=2^{\log(\dc(X))\cdot \w\cdot sm(X)}=2^{\w\cdot iw(X)}.$$
\end{proof}

\subsection{Pre-uniform spaces and their cardinal characteristics}\label{s13}

By an {\em entourage} on a set $X$ we understand any subset $U\subset X\times X$ containing the diagonal $\Delta_X=\{(x,y)\in X\times X:x=y\}$ of $X\times X$. For an entourage $U$ on $X$, point $x\in X$ and  subset $A\subset X$ let $B(x;U)=\{y\in X:(x,y)\in U\}$ be the {\em $U$-ball} centered at $x$, and $B(A;U)=\bigcup_{a\in A}B(a;U)$ be the {\em $U$-neighborhood} of $A$ in $X$.

Now we define some operations on entourages.
For two entourages $U,V$ on $X$ let $$U^{-1}=\{(x,y)\in X\times X:(y,x)\in U\}$$ be the {\em inverse} entourage and $$UV=\{(x,z)\in X\times X:\exists y\in X\;\mbox{ such that $(x,y)\in U$ and $(y,z)\in V$}\}$$be the {\em composition} of $U$ and $V$. It is easy to see that $(UV)^{-1}=V^{-1}U^{-1}$. For every entourage $U$ on $X$ define its powers $U^n$, $n\in\IZ$, by the formula: $U^0=\Delta_X$ and $U^{n+1}=U^nU$, $U^{-n-1}=U^{-n}U^{-1}$ for $n\in\w$. Define also the {\em alternating powers} $U^{\pm n}$ and $U^{\mp n}$ of $U$ by the recursive formulas: $U^{\pm0}=U^{\mp0}=\Delta_X$, and $U^{\pm(n+1)}=UU^{\mp n}$, $U^{\mp(n+1)}=U^{-1}U^{\pm n}$ for $n\ge 0$. If $U$ is an entourage on a topological space $X$, then put $\overline{U}=\bigcup_{x\in X}\overline{B(x;U)}$ be the closure of $U$ in the product $X_d\times X$ where $X_d$ is the set $X$ endowed with the discrete topology.
\smallskip

The following lemma proved in \cite{BRv} shows that the alternating power $U^{\mp2}$ on an entourage $U$ is equivalent to taking the star with respect to the cover $\U=\{B(x;U):x\in X\}$.

\begin{lemma}\label{l2.2} For any entourage $U$ on a set $X$ and a point $x\in X$ we get $B(x;U^{-1}U)=\St(x;\U)$ where $\U=\{B(x;U):x\in X\}$. Consequently, $B(x;U^{\mp 2n})=B(x;(U^{-1}U)^n)=\St^n(x;\U)$ for every $n\in\IN$.
\end{lemma}

A family $\U$ of entourages on a set $X$ is called a {\em uniformity} on $X$ if it satisfies the following four axioms:
\begin{itemize}
\item[(U1)] for any $U\in\U$, every entourage $V\subset X\times X$ containing $U$ belongs to $\U$;
\item[(U2)] for any entourages $U,V\in\U$ there is an entourage $W\in\U$ such that $W\subset U\cap V$;
\item[(U3)] for any entourage $U\in\U$ there is an entourage $V\in\U$ such that $VV\subset U$;
\item[(U4)] for any entourage $U\in\U$ there is an entourage $V\in\U$ such that $V\subset U^{-1}$.
\end{itemize}
A family $\U$ of entourages on $X$ is called a {\em quasi-uniformity} (resp. {\em pre-uniformity}) on $X$ if it satisfies the axioms (U1)--(U3) (resp. (U1)--(U2)~).
So, each uniformity is a quasi-uniformity and each quasi-uniformity is a pre-uniformity. Observe that a pre-uniformity is just a filter of entourages on $X$.

A subfamily $\mathcal B\subset\U$ is called a {\em base} of a pre-uniformity $\U$ on $X$ if each entourage $U\in\U$ contains some entourage $B\in\mathcal B$. Each base of a preuniformity satisfies the axiom (U2). Conversely, each family  $\mathcal B$ of entourages on $X$ satisfying the axiom $(U2)$ is a base of a unique pre-uniformity $\langle\mathcal B\rangle$ consisting of entourages $U\subset X\times X$ containing some entourage $B\in\mathcal B$. If the base $\mathcal B$ satisfies the axiom (U3) (and (U4)), then the pre-uniformity $\langle\mathcal B\rangle$ is a quasi-uniformity (and a uniformity).

Next we define some operations over preuniformities.
Given two preuniformities $\U,\V$ on a set $X$ put $\U^{-1}=\{U^{-1}:U\in\U\}$, $\U\wedge\V=\{U\cup V:U\in\U,\;V\in\V\}$, $\U\vee\V=\{U\cap V:U\in\U,\;V\in\V\}$ and let $\U\V$ be the pre-uniformity generated by the base $\{UV:U\in\U,\;V\in\V\}$. For every $n\in\w$ let $\U^{\pm n}$, $\U^{\mp n}$, $\U^{\wedge n}$, $\U^{\vee n}$ be the pre-uniformities generated by the bases $\{U^{\pm n}:U\in\U\}$, $\{U^{\mp n}:U\in\U\}$, $\{U^{\pm n}\cup U^{\mp n}:U\in\U\}$, $\{U^{\pm n}\cap U^{\mp n}:U\in \U\}$, respectively. Observe that $\U^{\wedge n}=\U^{\pm n}\wedge \U^{\mp n}$ and $\U^{\vee n}=\U^{\pm n}\vee \U^{\mp n}$. For a pre-uniformity $\U$ on a topological space $X$ let $\overline{\U}$ be the pre-uniformity generated by the base $\{\overline{U}:U\in\U\}$.

The pre-uniformities $\U^{\pm n}$, $\U^{\mp n}$, $\U^{\wedge n}$, $\U^{\vee n}$ feet into the following diagram (in which an arrow $\V\to\W$ indicates that $\V\subset\W$):
$$
\xymatrix{
&&\U^{\pm n}\ar[rd]\\
\U^{\vee(n+1)}\ar[r]&\U^{\wedge n}\ar[ru]\ar[rd]&&\U^{\vee n}\ar[r]&\U^{\wedge(n-1)}\\
&&\U^{\mp n}\ar[ru]
}
$$

We shall say that a preuniformity $\U$ on $X$ is
\begin{itemize}
\item  {\em $\pm n$-separated} if $\bigcap\U^{\pm n}=\Delta_X$;
\item  {\em $\mp n$-separated} if $\bigcap\U^{\mp n}=\Delta_X$;
\item {\em $n$-separated} if $\U$ is both $\pm n$-separated and $\mp n$-separated.
\end{itemize}
Observe that for an odd number $n$ a pre-uniformity $\U$ is $n$-separated if and only if it is $\pm n$-separated if and only if it is $\mp n$-separated (this follows from the equality $(U^{\pm n})^{-1}=U^{\mp n}$ holding for every entourage $U$).

This equivalence does not hold for even $n$:

\begin{example} For every $m\in\IN$ consider the entourage $U_m=\{(x,y)\in\IR_+\times\IR_+:y\in\{x\}\cup[x+m,\infty)\}$ on the half-line $\IR_+=[0,\infty)$. The family $\{U_m\}_{m\in\IN}$ is a base of a quasi-uniformity $\U$ on $\IR_+$ which is $\mp2$-separated but not $\pm2$-separated.
\end{example}

Each pre-uniformity $\U$ on a set $X$ generates a topology $\tau_\U$ consisting of all subsets $W\subset X$ such that for each point $x\in W$ there is an entourage $U\in\U$ with $B(x;U)\subset W$. This topology $\tau_\U$ will be referred to as {\em the topology generated by the pre-uniformity} $\U$. If $\U$ is a quasi-uniformity, then for each point $x\in X$ the family of balls $\{B(x;U):U\in\U\}$ is a neighborhood base of the topology $\tau_\U$ at $x$. This implies that for a quasi-uniformity $\U$ on a set $X$ the topology $\tau_\U$ is Hausdorff if and only if  for any distinct points $x,y\in X$ there is an entourage $U\in\U$ such that $B(x;U)\cap B(y;U)=\emptyset$ if and only if $\bigcap\U\U^{-1}=\Delta_X$ if and only if the quasi-uniformity $\U$ is $\pm2$-separated.
It is known (see \cite{Ku1} or \cite{Ku2}) that the topology of each topological space $X$ is generated by a suitable quasi-uniformity (in particular, the Pervin quasi-uniformity, generated by the subbase consisting of the entourages $(U\times U)\cup \big((X\setminus U)\times X\big)$ where $U$ runs over open sets in $X$).

Now we consider some cardinal characteristics of pre-uniformities. Let $\U$ be a pre-uniformity on a topological space $X$.
\begin{itemize}
\item The {\em boundedness number} $\ell(\U)$ of $\U$ is defined as the smallest cardinal $\kappa$ such that for any entourage $U\in\U$ there is a subset $A\subset X$ of cardinality $|A|\le \kappa$ such that $B(A;U)=X$;
\item the {\em weak boundedness number} $\bar\ell(\U)$ of $\U$ is defined as the smallest cardinal $\kappa$ such that for any entourage $U\in\U$ there is a subset $A\subset X$ of cardinality $|A|\le \kappa$ such that $B(A;U)$ is dense in $X$;
\smallskip
\item the {\em character} $\chi(\U)$ of $\U$ is the smallest cardinality of a subfamily $\V\subset\U$ such that each entourage $U\in\U$ contains some entourage $V\in\V$;
\item the {\em pseudocharacter} $\psi(\U)$ of $\U$ is the smallest cardinality of a subfamily $\V\subset\U$ such that $\bigcap\V=\bigcap\U$;
\item the {\em closed pseudocharacter} $\overline{\psi}(\U)$ of $\U$ is the smallest cardinality of a subfamily $\V\subset\U$ such that for every $x\in X$ we get $\bigcap_{V\in\V}\overline{B(x;V)}=\bigcap_{U\in\U}\overline{B(x;U)}$ (so, $\overline{\psi}(\U)=\psi(\overline{\U})$~);
\item the {\em local pseudocharacter} $\dot\psi(\U)$ of $\U$ is the smallest cardinal $\kappa$ such that for every $x\in X$ there is a subfamily $\V_x\subset\U$ of cardinality $|\V_x|\le\kappa$ such that $\bigcap_{V\in\V_x}B(x;V)=\bigcap_{U\in\U}B(x;U)$.
\end{itemize}

For any Hausdorff topological space $X$ and a quasi-uniformity $\U$ generating the topology of $X$ we get the inequalities $\psi(X)=\dot\psi(\U)\le\psi(\U)$, $\overline{\psi}(X)\le\overline{\psi}(\U)$ and $\chi(X)\le\chi(\U)$, which fit into the following diagram (in which an arrow $a\to b$ indicates that $a\le b$).
$$\xymatrix{
\psi(X)\ar[r]\ar[d]&\overline{\psi}(X)\ar[r]\ar[d]&\chi(X)\ar[d]\\
\psi(\U)\ar[r]&\overline{\psi}(\U)\ar[r]&\chi(\U)}
$$

The boundedness number $\ell(\U)$ combined with the pseudocharacter $\psi^{\mp2}(\U)$ can be used to produce a simple upper bound on the cardinality of a $\mp2$-separated pre-uniform space (cf. \cite[4.3]{BBR}).

\begin{proposition}\label{card} Any set $X$ has cardinality $|X|\le\ell(\U)^{\psi^{\mp2}(\U)}$ for any $\mp2$-separated pre-uniformity $\U$ on a set $X$.
\end{proposition}

\begin{proof} The pre-uniformity $\U^{\mp2}$, being separated, contains a subfamily $\V\subset\U$ of cardinality $|\V|=\psi(\U^{\mp2})$ such that $\bigcap_{V\in\V}V^{-1}V=\Delta_X$. By the definition of the boundedness number $\ell(\U)$, for every entourage $V\in\V$ there is a subset $L_V\subset X$ of cardinality $|L_V|\le\ell(\U)$ such that $X=B(L_V;V)$. For every $x\in X$ choose a function $f_x\in\prod_{V\in\V}L_V$ assigning to every entourage $V\in\V$ a point $f_x(V)\in L_V$ such that $x\in B(f_x(V);V)$. We claim that for any distinct points $x,y\in X$ the functions $f_x,f_y$ are distinct. Indeed, the choice of the family $\V$ yields an entourage $V\in\V$ such that $(x,y)\notin V^{-1}V$. Then $f_x(V)\ne f_y(V)$ and hence $f_x\ne f_y$. This implies that $$|X|\le\prod_{V\in\V}|L_V|\le \ell(X)^{|\V|}=\ell(\U)^{\psi(\U^{\mp2})}.$$
\end{proof}

Following \cite{BR} we define a quasi-uniformity $\U$ on a topological space $X$ to be {\em normal} if for any subset $A\subset X$ and entourage $U\in\U$ we get $\overline{A}\subset \overline{B(A;U)}^\circ$. A topological space $X$ is called {\em normally quasi-uniformizable} if the topology of $X$ is generated by a normal quasi-uniformity.
Normally quasi-uniformizable spaces possess the following important normality-type property proved in \cite{BR}.

\begin{theorem}\label{t1.4} Let $X$ be a topological space and $\U$ be a normal quasi-uniformity generating the topology of $X$. Then for every subset $A\subset X$ and entourage $U\in\U$ there exists a continuous function $f:X\to [0,1]$ such that $A\subset f^{-1}(0)$ and $f\big([0,1)\big)\subset \overline{B(A;U)}^\circ$.
\end{theorem}

\subsection{Cardinal characteristics of topological spaces, II}\label{s14}

Let $X$ be a topological space. An entourage $U$ on $X$ is called a {\em neighborhood assignment} if for every $x\in X$ the $U$-ball $B(x;U)$ is a neighborhood of $x$. The family $\pU_X$ of all neighborhood assignments on a topological space $X$ is a pre-uniformity called the {\em universal pre-uniformity} on $X$. It contains any pre-uniformity generating the topology of $X$ and is equal to the union of all pre-uniformities generating the topology of $X$.

The universal pre-uniformity $\pU_X$ contains
\begin{itemize}
\item the {\em universal quasi-uniformity} $\qU_X=\bigcup\{\U\subset\pU_X:\U$ is a quasi-uniformity on $X\}$, and
\item the {\em universal uniformity} $\U_X=\bigcup\{\U\subset\pU_X:\U$ is a uniformity on $X\}$
\end{itemize}
of $X$.
It is clear that $\U_X\subset q\U_X\subset \pU_X$. The interplay between the universal pre-uniformities $\pU_X$, $\qU_X$ and $\U_X$ are studied in \cite{BRv}.

Since the topology of any topological space is generated by a quasi-uniformity, the universal quasi-uniformity $\qU_X$ generates the topology of $X$. In contrast, the universal uniformity $\U_X$ generates the topology of $X$ if and only if the space $X$ is completely regular.

Cardinal characteristics of the pre-uniformities $\pU_X$, $\qU_X$ and $\U_X$ or their alternating powers can be considered as cardinal characteristics of the topological space $X$. In particular, for a Hausdorff space $X$ we have the equalities:
$$\chi(X)=\chi(\pU_X),\;\;\psi(X)=\psi(\pU_X),\;\;\overline{\psi}(X)=\overline{\psi}(\pU_X),\;\;\Delta(X)=\psi(\pU_X^{\mp2}).$$
The last equality follows from Lemma~\ref{l2.2}.
On the other hand, the boundedness number $\ell(\pU_X)$ of $\pU_X$ coincides with the Lindel\"of number $l(X)$ of $X$.

Observe that for the universal pre-uniformity $\pU_X$ on a Hausdorff topological space $X$ the upper bound $|X|\le\ell(\pU_X)^{\psi(\pU_X^{\mp2})}$ proved in Proposition~\ref{card} turns into the known upper bound $|X|\le l(X)^{\Delta(X)}$.

Having in mind the equality  $l(X)=\ell(\pU_X)$, for every $n\in\IN$ let us define the following cardinal characteristics:
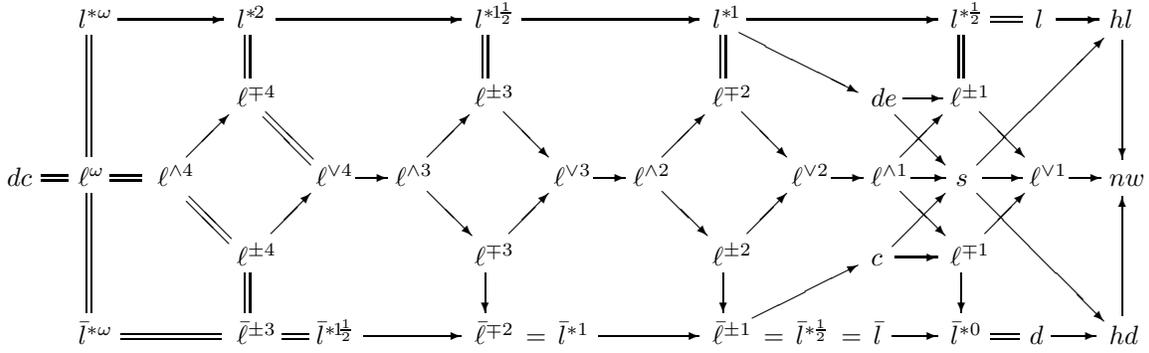
$$
\begin{aligned}
%&\Delta^{\pm n}(X):=\psi(\pU_X^{\pm n}),& &\Delta^{\mp n}(X):=\psi(\pU_X^{\mp n}),& &\Delta^{\wedge n}(X):=\psi(\pU_X^{\wedge n})& &\Delta^{\vee n}(X):=\psi(\pU_X^{\vee n}),&\\
&\ell^{\pm n}(X):=\ell(\pU_X^{\pm n}),& &\ell^{\mp n}(X):=\ell(\pU_X^{\mp n}),& &\ell^{\wedge n}(X):=\ell(\pU_X^{\wedge n})& &\ell^{\vee n}(X):=\ell(\pU_X^{\vee n}),&\\
&\bar\ell^{\pm n}(X):=\bar\ell(\pU_X^{\pm n}),& &\bar\ell^{\mp n}(X):=\bar\ell(\pU_X^{\mp n}),&
&\bar\ell^{\wedge n}(X):=\bar \ell(\pU_X^{\wedge n}),&
&\bar\ell^{\vee n}(X):=\bar \ell(\pU_X^{\vee n}),&\\
&q\ell^{\pm n}(X):=\ell(\qU_X^{\pm n}),& &q\ell^{\mp n}(X):=\ell(\qU_X^{\mp n}),& &q\ell^{\wedge n}(X):=\ell(\qU_X^{\wedge n}),& &q\ell^{\vee n}(X):=\ell(\qU_X^{\vee n}).&
\end{aligned}
$$
 Let also $$\ell^{\w}(X)=\min_{n\in\IN}\ell^{\vee n}(X),\quad q\ell^{\w}(X)=\min_{n\in\IN}q\ell^{\vee n}(X), \mbox{ \ \ and \ \ }u\ell(X)=\ell(\U_X).$$
Observe that $u\ell(X)=\ell(\U_X^{\pm n})=\ell(\U_X^{\mp n})=\ell(\U_X^{\wedge n})=\ell(\U_X^{\vee n})$ for every $n\in\IN$ (this follows from the equality $\U_X=\U_X^{\pm n}=\U_X^{\mp n}$ holding for every $n\in\IN$).
}

The above cardinal characteristics were introduced and studied in \cite{BRv}.

Some inequalities between the cardinal characteristics $\ell^{\pm n}$, $\ell^{\mp n}$, $\ell^{\wedge n}$, $\ell^{\vee n}$, $q\ell^{\pm n}$, $q\ell^{\mp n}$,  $q\ell^{\wedge n}$, $q\ell^{\vee n}$, and $u\ell$ are described in the following diagram in which an arrow $a\to  b$ indicates that $a(X)\le b(X)$ for any topological space $X$.
$$\xymatrix{
&&&\ell^{\vee n}\\
&&&q\ell^{\vee n}\ar[u]\\
\bar\ell^{\mp n}\ar[r]&\ell^{\mp n}\ar[rruu]&q\ell^{\mp n}\ar[l]\ar[ru]&u\ell\ar[d]\ar[u]\ar[l]\ar[r]&q\ell^{\pm n}\ar[r]\ar[lu]&\ell^{\pm n}\ar[lluu]&\bar\ell^{\pm n}\ar[l]\\
&&&q\ell^{\wedge n}\ar[d]\ar[ru]\ar[lu]\\
&&&\ell^{\wedge n}\ar[rruu]\ar[lluu]
}
$$
It turns out that the cardinal invariants $l^{*n}$, $l^{*n\frac12}$, $\bar l^{*n}$, and $\bar l^{*n\frac12}$ can be expressed via the cardinal invariants $\ell^{\mp m}$, $\ell^{\pm m}$,  $\ell^{\mp m}$, $\ell^{\pm m}$ for a suitable number $m$. The following proposition is proved in \cite{BRv} (or can be easily derived from the definitions).

\begin{proposition}\label{p1.5} For every $n\in\w$ we have the equalities:
$$l^{*n}=\ell^{\mp 2n},\;\;\bar l^{*n}=\bar\ell^{\mp 2n},\;\;l^{*n\kern-1pt\frac12}=\ell^{\pm(2n+1)},\;\;\bar l^{*n\kern-1pt\frac12}=\bar\ell^{\pm(2n+1)}.$$
\end{proposition}

The following proposition (proved in \cite{BRv}) describes the relation of the cardinal invariants $\ell^{\pm n}$, $\ell^{\mp n}$ to classical cardinal invariants.

\begin{proposition}\label{p1.8} Let $X$ be a topological space. Then
\begin{enumerate}
\item $\ell^{\wedge 1}(X)\le s(X)\le q\ell^{\vee1}(X)\le \ell^{\vee 1}(X)\le nw(X)$;
\item $e(X)\le de(X)\le q\ell^{\pm1}(X)\le \ell^{\pm1}(X)=l(X)$;
\item $c(X)\le q\ell^{\mp1}(X)\le \ell^{\mp1}(X)\le d(X)$;
\item If $X$ is quasi-regular, then $\bar\ell^{\pm 3}(X)=\bar l^{*1\kern-1pt\frac12}(X)=\ell^{\w}(X)=\dc(X)$;
\item If $X$ is completely regular, then $q\bar\ell^{\pm 3}(X)=q\ell^{\w}(X)=u\ell(X)=\dc(X)$.
\end{enumerate}
\end{proposition}

Taking into account Propositions~\ref{p1.2}, \ref{p1.5} and \ref{p1.8}, we see that for quasi-regular spaces the cardinal characteristics $\ell^{\pm n}$, $\ell^{\mp n}$,  $\bar\ell^{\mp n}$, $\ell^{\wedge n}$, $\ell^{\vee n}$ relate to other cardinal characteristics of topological spaces as follows.

\begin{picture}(400,140)(-35,-10)
\put(0,0){$\bar l^{*\w}$}
\put(3,12){\line(0,1){45}}
\put(5,12){\line(0,1){45}}
\put(16,2){\line(1,0){38}}
\put(16,4){\line(1,0){38}}
\put(60,0){$\bar\ell^{\pm3}$}
\put(77,2){\line(1,0){10}}
\put(77,4){\line(1,0){10}}
\put(63,12){\line(0,1){15}}
\put(65,12){\line(0,1){15}}
\put(90,0){$\bar l^{*\kern-1pt 1\kern-2pt \frac12}$}
\put(108,3){\vector(1,0){37}}
\put(150,0){$\bar\ell^{\mp 2}\,=\,\bar l^{*1}$}
\put(197,3){\vector(1,0){38}}
\put(240,0){$\bar\ell^{\pm 1}\,=\,\bar l^{*\kern-1pt \frac12}\,=\,\bar l$}
\put(255,10){\vector(2,1){40}}
\put(308,3){\vector(1,0){17}}
\put(330,0){$\bar l^{*0}$}
\put(345,2){\line(1,0){11}}
\put(345,4){\line(1,0){11}}
\put(360,0){$d$}
\put(368,3){\vector(1,0){17}}
\put(390,0){$hd$}
\put(395,11){\vector(0,1){45}}

\put(0,120){$l^{*\w}$}
\put(15,123){\vector(1,0){40}}
\put(60,120){$l^{*\kern-1pt 2}$}
\put(75,123){\vector(1,0){70}}
\put(150,120){$l^{*\kern-1pt 1\kern-1.5pt \frac12}$}
\put(168,123){\vector(1,0){67}}
\put(240,120){$l^{*\kern-1pt 1}$}
\put(250,118){\vector(2,-1){45}}
\put(253,123){\vector(1,0){72}}
\put(330,120){$l^{*\kern-1pt \frac12}$}
\put(345,122){\line(1,0){12}}
\put(345,124){\line(1,0){12}}
\put(362,120){$l$}
\put(370,123){\vector(1,0){16}}
\put(390,120){$hl$}
\put(395,115){\vector(0,-1){45}}

%\put(0,30){$dc$}
%\put(3,40){\line(0,1){17}}
%\put(5,40){\line(0,1){17}}
\put(60,30){$\ell^{\pm4}$}
\put(73,42){\vector(1,1){15}}
\put(150,30){$\ell^{\mp 3}$}
\put(163,42){\vector(1,1){15}}
\put(154,27){\vector(0,-1){15}}
\put(240,30){$\ell^{\pm2}$}
\put(253,42){\vector(1,1){15}}
\put(244,27){\vector(0,-1){15}}
\put(300,30){$c$}
\put(308,37){\vector(1,1){20}}
\put(309,33){\vector(1,0){16}}
\put(330,30){$\ell^{\mp1}$}
\put(343,42){\vector(1,1){15}}
\put(334,27){\vector(0,-1){15}}

%\put(0,90){$u\ell$}
%\put(3,101){\line(0,1){16}}
%\put(5,101){\line(0,1){16}}
\put(60,90){$\ell^{\mp4}$}
\put(71,88){\line(1,-1){18}}
\put(70,86){\line(1,-1){18}}
%\put(71,88){\vector(1,-1){17}}
\put(63,101){\line(0,1){16}}
\put(65,101){\line(0,1){16}}
\put(150,90){$\ell^{\pm 3}$}
\put(161,88){\vector(1,-1){17}}
\put(153,101){\line(0,1){16}}
\put(155,101){\line(0,1){16}}
\put(240,90){$\ell^{\mp2}$}
\put(251,88){\vector(1,-1){17}}
\put(243,101){\line(0,1){16}}
\put(245,101){\line(0,1){16}}
\put(300,90){$de$}
\put(309,87){\vector(1,-1){19}}
\put(312,93){\vector(1,0){15}}
\put(330,90){$\ell^{\pm1}$}
\put(341,88){\vector(1,-1){17}}
\put(333,101){\line(0,1){16}}
\put(335,101){\line(0,1){16}}

\put(-27,60){$dc$}
\put(-14,62){\line(1,0){10}}
\put(-14,64){\line(1,0){10}}
\put(0,60){$\ell^{\w}$}
\put(3,71){\line(0,1){46}}
\put(5,71){\line(0,1){46}}
\put(12,62){\line(1,0){12}}
\put(12,64){\line(1,0){12}}
\put(30,60){$\ell^{{\wedge}4}$}
\put(42,56){\line(1,-1){16}}
\put(41,54){\line(1,-1){16}}
\put(41,71){\vector(1,1){16}}
\put(90,60){$\ell^{\vee4}$}
\put(105,63){\vector(1,0){12}}
\put(120,60){$\ell^{{\wedge}3}$}
\put(132,71){\vector(1,1){16}}
\put(132,58){\vector(1,-1){17}}
\put(180,60){$\ell^{\vee3}$}
\put(195,63){\vector(1,0){12}}
\put(210,60){$\ell^{{\wedge}2}$}
\put(221,71){\vector(1,1){16}}
\put(221,58){\vector(1,-1){17}}
\put(270,60){$\ell^{\vee2}$}
\put(285,63){\vector(1,0){12}}
\put(300,60){$\ell^{{\wedge}1}$}
\put(311,71){\vector(1,1){16}}
\put(311,58){\vector(1,-1){17}}
\put(315,63){\vector(1,0){13}}
\put(332,60){$s$}
\put(340,68){\vector(1,1){48}}
\put(340,57){\vector(1,-1){48}}
\put(342,63){\vector(1,0){15}}
\put(360,60){$\ell^{\vee1}$}
\put(375,63){\vector(1,0){12}}
\put(390,60){$nw$}
\end{picture}

For Tychonoff spaces we can add to this diagram the cardinal characteristics $q\ell^{\pm n}$, $q\ell^{\mp n}$, $q\ell^{\vee n}$, and $u\ell$:

{\small
\begin{picture}(400,180)(5,-2)
\put(1,0){$\bar l^{*\w}$}
\put(3,12){\line(0,1){15}}
\put(5,12){\line(0,1){15}}
\put(15,2){\line(1,0){45}}
\put(15,4){\line(1,0){45}}
\put(63,0){$\bar\ell^{\pm3}$}
\put(80,2){\line(1,0){15}}
\put(80,4){\line(1,0){15}}
\put(65,12){\line(0,1){15}}
\put(67,12){\line(0,1){15}}
\put(100,0){$\bar l^{*\kern-1pt 1\kern-1pt \frac12}$}
\put(115,3){\vector(1,0){45}}
\put(164,0){$\bar\ell^{\mp 2}$}
\put(180,2){\line(1,0){16}}
\put(180,4){\line(1,0){16}}
\put(200,0){$\bar l^{*\kern-1pt 1}$}
\put(212,3){\vector(1,0){48}}
\put(264,0){$\bar\ell^{\pm 1}$}
\put(280,2){\line(1,0){16}}
\put(280,4){\line(1,0){16}}
\put(300,0){$\bar l^{*\kern-1pt \frac12}$}
\put(314,2){\line(1,0){12}}
\put(314,4){\line(1,0){12}}
\put(330,0){$\bar l$}
\put(337,3){\vector(1,0){38}}
\put(279,7){\vector(2,1){46}}

\put(379,0){$\bar l^{*0}$}
\put(392,2){\line(1,0){22}}
\put(392,4){\line(1,0){22}}
\put(418,0){$d$}
\put(426,3){\vector(1,0){25}}
\put(455,0){$hd$}
\put(460,11){\vector(0,1){65}}

\put(0,30){$dc$}
\put(3,40){\line(0,1){37}}
\put(5,40){\line(0,1){37}}
\put(63,30){$\ell^{\pm4}$}
\put(73,41){\vector(3,4){27}}
\put(163,30){$\ell^{\mp 3}$}
\put(173,41){\vector(3,4){27}}
\put(166,27){\vector(0,-1){16}}
\put(263,30){$\ell^{\pm2}$}
\put(273,41){\vector(3,4){27}}

\put(266,27){\vector(0,-1){16}}
\put(329,31){$c$}
\put(336,37){\vector(1,1){41}}
\put(337,33){\vector(1,0){38}}
\put(378,30){$\ell^{\mp1}$}
\put(387,41){\vector(1,1){35}}
\put(380,27){\vector(0,-1){15}}

\put(62,55){$q\ell^{\pm4}$}
\put(70,64){\vector(3,4){9}}
\put(65,51){\line(0,-1){11}}
\put(67,51){\line(0,-1){11}}
\put(162,55){$q\ell^{\mp 3}$}
\put(170,64){\vector(3,4){9}}
\put(166,51){\vector(0,-1){11}}
\put(263,55){$q\ell^{\pm2}$}
\put(270,64){\vector(3,4){9}}
\put(266,51){\vector(0,-1){11}}
\put(375,55){$q\ell^{\mp1}$}
\put(387,65){\vector(1,1){10}}
\put(380,51){\vector(0,-1){11}}

\put(2,80){$\ell^{\w}$}
\put(3,90){\line(0,1){36}}
\put(5,90){\line(0,1){36}}
\put(12,82){\line(1,0){10}}
\put(12,84){\line(1,0){10}}
\put(25,80){$\ell^{\wedge\kern-1pt 4}$}
\put(32,90){\vector(3,4){28}}
\put(32,79){\line(3,-4){29}}
\put(31,77){\line(3,-4){29}}
\put(50,80){$q\ell^{\wedge\kern-1pt 4}$}
\put(53,88){\vector(3,4){10}}
\put(54,76){\line(3,-4){10}}
\put(52,75){\line(3,-4){10}}
\put(47,83){\line(-1,0){9}}
\put(47,81){\line(-1,0){9}}
\put(75,80){$q\ell^{\vee\kern-1pt 4}$}
\qbezier(88,92)(115,120)(139,97)
\put(139,97){\vector(1,-1){10}}
\put(90,82){\vector(1,0){8}}
\put(100,80){$\ell^{\vee\kern-1pt 4}$}
\put(112,82){\vector(1,0){10}}
\put(125,80){$\ell^{\wedge\kern-1pt 3}$}
\put(132,90){\vector(3,4){28}}
\put(132,78){\vector(3,-4){29}}
\put(150,80){$q\ell^{\wedge\kern-1pt 3}$}
\put(153,88){\vector(3,4){10}}
\put(154,76){\vector(3,-4){10}}
\put(148,82){\vector(-1,0){10}}
\put(175,80){$q\ell^{\vee\kern-1pt 3}$}
\put(190,82){\vector(1,0){8}}
\put(200,80){$\ell^{\vee\kern-1pt 3}$}
\put(212,82){\vector(1,0){10}}
\qbezier(188,78)(213,48)(238,67)
\put(238,67){\vector(1,1){10}}

\put(225,80){$\ell^{\wedge\kern-1pt 2}$}
\put(232,90){\vector(3,4){28}}
\put(232,78){\vector(3,-4){29}}
\put(250,80){$q\ell^{\wedge\kern-1pt 2}$}
\put(253,88){\vector(3,4){10}}
\put(254,76){\vector(3,-4){10}}
\put(248,82){\vector(-1,0){10}}
\put(275,80){$q\ell^{\vee\kern-1pt 2}$}
\qbezier(288,92)(315,120)(339,97)
\put(339,97){\vector(1,-1){10}}
\put(290,82){\vector(1,0){8}}
\put(300,80){$\ell^{\vee\kern-1pt 2}$}
\put(312,82){\vector(1,0){10}}
\put(325,80){$\ell^{\wedge\kern-1pt 1}$}
\put(338,90){\vector(1,1){37}}
\put(338,76){\vector(1,-1){37}}
\put(350,80){$q\ell^{\wedge\kern-1pt 1}$}
\put(360,90){\vector(1,1){13}}
\put(359,77){\vector(1,-1){15}}
\put(348,82){\vector(-1,0){10}}
\put(365,82){\vector(1,0){10}}
\put(378,80){$s$}
\put(384,85){\vector(1,1){70}}
\put(384,79){\vector(1,-1){70}}
\put(385,82){\vector(1,0){8}}
\put(395,80){$q\ell^{\vee\kern-1pt 1}$}
\put(410,82){\vector(1,0){8}}
\put(420,80){$\ell^{\vee\kern-1pt 1}$}
\put(432,82){\vector(1,0){20}}
\put(455,80){$nw$}

\put(62,105){$q\ell^{\mp4}$}
\put(70,103){\line(3,-4){11}}
\put(69,101){\line(3,-4){10}}
\put(66,114){\vector(0,1){13}}
\put(162,105){$q\ell^{\pm 3}$}
\put(170,102){\vector(3,-4){9}}
\put(166,114){\vector(0,1){13}}
\put(263,105){$q\ell^{\mp2}$}
\put(270,102){\vector(3,-4){9}}
\put(266,114){\vector(0,1){13}}
\put(375,105){$q\ell^{\pm1}$}
\put(384,102){\vector(1,-1){13}}
\put(380,114){\vector(0,1){13}}

\put(0,130){$u\ell$}
\put(3,138){\line(0,1){19}}
\put(5,138){\line(0,1){19}}
\put(3,40){\line(0,1){37}}
\put(5,40){\line(0,1){37}}
\put(63,130){$\ell^{\mp4}$}
\put(65,140){\line(0,1){17}}
\put(67,140){\line(0,1){17}}
\put(71,129){\line(3,-4){30}}
\put(70,127){\line(3,-4){29}}
%\put(71,128){\vector(3,-4){29}}
\put(163,130){$\ell^{\pm 3}$}
\put(165,140){\line(0,1){17}}
\put(167,140){\line(0,1){17}}

\put(171,128){\vector(3,-4){29}}
%\put(166,27){\vector(0,-1){16}}
\put(263,130){$\ell^{\mp2}$}
\put(265,140){\line(0,1){17}}
\put(267,140){\line(0,1){17}}

\put(271,128){\vector(3,-4){29}}

%\put(266,27){\vector(0,-1){16}}
\put(329,129){$de$}
\put(337,126){\vector(1,-1){40}}
\put(340,132){\vector(1,0){35}}
\put(378,130){$\ell^{\pm1}$}
\put(379,140){\line(0,1){17}}
\put(381,140){\line(0,1){17}}

\put(384,127){\vector(1,-1){36}}
%\put(380,27){\vector(0,-1){15}}

\put(1,160){$l^{*\w}$}
\put(15,163){\vector(1,0){45}}
\put(64,160){$l^{*\kern-1pt 2}$}
\put(76,163){\vector(1,0){84}}
\put(164,160){$l^{*\kern-1pt 1\kern-1pt \frac12}$}
\put(182,163){\vector(1,0){78}}
\put(264,160){$l^{*\kern-1pt 1}$}
\put(275,158){\vector(2,-1){50}}
\put(277,163){\vector(1,0){97}}
\put(378,160){$l^{*\kern-1pt \frac12}$}
\put(392,162){\line(1,0){22}}
\put(392,164){\line(1,0){22}}
\put(418,160){$l$}
\put(426,163){\vector(1,0){25}}
\put(455,160){$hl$}
\put(460,156){\vector(0,-1){67}}
\end{picture}
}

\begin{question} Which cardinal characteristics in the above diargams are pairwise distinct?
\end{question}

\section{$i$-Weight of normally quasi-uniformizable topological spaces}\label{s2}

In this section we apply Theorem~\ref{t1.4} to derive some upper bounds on the $i$-weight of a normally quasi-uniformizable space.

\begin{proposition}\label{p2.1} Let $X$ be a topological space whose topology is generated by a normal quasi-uniformity $\U$. The space $X$ has $i$-weight $iw(X)\le\kappa$ for some cardinal $\kappa$ if there exists a family of subsets $\{A_\alpha\}_{\alpha\in\kappa}$ of $X$ and a family of entourages $\{U_\alpha\}_{\alpha\in\kappa}\subset\U$ such that for any distinct points $x,y\in X$ there is $\alpha\in\kappa$ such that $x\in A_\alpha$ and $y\notin\overline{B(A_\alpha;U_\alpha)}$.
\end{proposition}

\begin{proof} For every $\alpha\in\kappa$ apply Theorem~\ref{t1.4} to construct a continuous map $f_\alpha:X\to[0,1]$ such that $f_\alpha(A_\alpha)\subset\{0\}$ and $f_\alpha^{-1}\big([0,1)\big)\subset \overline{B(A_\alpha;U_\alpha)}$. It follows that the family of continuous maps $\{f_\alpha\}_{\alpha\in\kappa}$ separates points of $X$. So, $iw(X)\le \kappa$.
\end{proof}

This proposition will be used to prove:

\begin{theorem}\label{t2.2} A Hausdorff space $X$ has $i$-weight $iw(X)\le \overline{\psi}(\A^{-1}\kern-2pt\A\U)\cdot \ell(\A)$ for any normal quasi-uniformity $\U$ generating the topology of $X$ and  any pre-uniformity $\A$ on $X$ such that $\bigcap\overline{\A^{-1}\kern-2pt\A\U}=\Delta_X$.
\end{theorem}

\begin{proof} If the cardinal $\overline{\psi}(\A^{-1}\A\U)$ is finite, then $\overline{\psi}(\A^{-1}\A\U)=1$, which implies that $A^{-1}AU=\Delta_X=A=U$ for some $A\in\A$ and $U\in\U$. In this case $\ell(\A)=|X|$ and hence $iw(X)\le|X|\le\ell(\A)$.

So, we assume that the cardinal $\kappa=\overline{\psi}(\A^{-1}\A\U)$ is infinite. Since $\bigcap\overline{\A^{-1}\A\U}=\Delta_X$, we can choose subfamilies $(A_\alpha)_{\alpha\in\kappa}\subset\A$ and $(U_\alpha)_{\alpha\in\kappa}\subset\U$ such that $\bigcap_{\alpha<\kappa}\overline{B(x,A_\alpha^{-1}A_\alpha U_\alpha)}=\{x\}$ for every $x\in X$. For every $\alpha\le\kappa$ choose a subset
$Z_\alpha\subset X$ of cardinality $|Z_\alpha|\le\ell(\A)$ such that $X=B(Z_\alpha;A_\alpha)$.
Consider the family of sets $\mathcal Z=\bigcup_{\alpha\in\kappa}\{B(z;A_\alpha):z\in Z_\alpha\}$.
We claim that for any distinct points $x,y\in X$ there is a set $Z\in\mathcal Z$ and ordinal $\alpha\in\kappa$ such that $x\in Z$ and $y\notin\overline{B(Z;U_\alpha)}$.

By the choice of the families $(A_\alpha)$, $(U_\alpha)$, for the points $x,y$ there is an index $\alpha\in\kappa$ such that $y\notin \overline{B(x;A_\alpha^{-1}A_\alpha U_\alpha)}$.
Since $X=B(Z_\alpha;A_\alpha)$, we can find a point $z\in Z_\alpha$ such that $x\in B(z;A_\alpha)$ and hence $z\in B(x;A_\alpha^{-1})$. We claim that the set $Z=B(z;A_\alpha)\in\mathcal Z$ has the required properties: $x\in Z$ and $y\notin\overline{B(Z;U_\alpha)}$. To derive a contradiction, assume that $y\in\overline{B(Z;U_\alpha)}$ which implies
$$y\in\overline{B(Z;U_\alpha)}=\overline{B(B(z;A_\alpha);U_\alpha)}= \overline{B(z;A_\alpha U_\alpha) }\subset \overline{B(B(x;A_\alpha^{-1});A_\alpha U_\alpha)}=
\overline{B(x;A_\alpha^{-1}A_\alpha U_\alpha)}.$$
But this contradicts the choice of the index $\alpha$.

This contradiction allows us to apply Proposition~\ref{p2.1} and conclude that
$$iw(X)\le |\mathcal Z|\cdot \kappa\le  \sum_{\alpha\in\kappa}|Z_\alpha|\cdot\kappa\le\kappa^2\cdot\ell(\A)=\overline{\psi}(\A^{-1}\A\U)\cdot\ell(\A).$$
\end{proof}

Applying Theorem~\ref{t2.2} to some concrete pre-uniformities $\A$, we get the following corollary.

\begin{corollary}\label{c2.3} Let $X$ be a functionally Hausdorff space and $\U$ be a normal quasi-uniformity generating the topology of $X$. If for some $n\in\IN$  the quasi-uniformity $\U$ is
\begin{enumerate}
\item $\pm(4n-2)$-separated, then  $iw(X)\le\overline{\psi}(\U^{\pm (4n-3)})\cdot\ell(\U^{\vee(2n-1)})\le \chi(\U)\cdot q\ell^{\vee(2n-1)}(X)$;
\item $\mp(4n-1)$-separated, then $iw(X)\le\overline{\psi}(\U^{\mp(4n-2)})\cdot\ell(\U^{\pm(2n-1)})\le \chi(\U)\cdot q\ell^{\pm(2n-1)}(X)$;
\item $\pm (4n)$-separated, then $iw(X)\le\overline{\psi}(\U^{\pm(4n-1)})\cdot\ell(\U^{\vee(2n)})\le \chi(\U)\cdot q\ell^{\vee(2n)}(X)$;
\item $\mp(4n+1)$-separated, then $iw(X)\le\overline{\psi}(\U^{\mp(4n)})\cdot\ell(\U^{\mp(2n)})\le \chi(\U)\cdot q\ell^{\mp(2n)}(X)$.
\end{enumerate}
\end{corollary}

\begin{proof}
1. If $\U$ is $\pm(4n-2)$-separated, then for the pre-uniformity $\A=\U^{\pm(2n-1)}\vee \U^{\mp(2n-1)}$ we get
$$\A^{-1}\A\U\subset\U^{\pm(2n-1)}\U^{\pm(2n-1)}\U=\U^{\pm(4n-3)}\U=\U^{\pm(4n-3)}$$ and hence $\bigcap\overline{\A^{-1}\A\U}\subset \bigcap \A^{-1}\A\U\U^{-1}=\bigcap \U^{\pm(4n-2)}=\Delta_X$.
Applying Theorem~\ref{t2.2} to the pre-uniformity $\A=\U^{\vee(2n-1)}$, we get
$$iw(X)\le\overline{\psi}(\U^{\pm(4n-3)})\cdot\ell(\U^{\vee(2n-1)})\le \chi(\U)\cdot q\ell^{\vee(2n-1)}(X).$$
\smallskip

2.  If $\U$ is $\mp(4n-1)$-separated, then for the pre-uniformity $\A=\U^{\pm(2n-1)}$ we get
$$\A^{-1}\A\U=\U^{\mp(2n-1)}\U^{\pm(2n-1)}\U=\U^{\mp(4n-2)}\U=\U^{\mp(4n-2)}$$ and hence $\bigcap\overline{\A^{-1}\A\U}\subset \bigcap \A^{-1}\A\U\U^{-1}=\bigcap\U^{\mp(4n-1)}=\Delta_X$.
Applying Theorem~\ref{t2.2} to the pre-uniformity $\A=\U^{\pm(2n-1)}$, we get
$$iw(X)\le\overline{\psi}(\U^{\mp(4n-2)})\cdot\ell(\U^{\pm(2n-1)})\le \chi(\U)\cdot q\ell^{\pm(2n-1)}(X).$$
\smallskip

3. If $\U$ is $\pm(4n)$-separated, then for the pre-uniformity $\A=\U^{\vee(2n)}$ we get
$$\A^{-1}\A\U\subset\U^{\pm(2n)}\U^{\mp(2n)}\U=\U^{\pm(4n-1)}\U=\U^{\pm(4n-1)}$$ and hence $\bigcap\overline{\A^{-1}\A\U}\subset \bigcap \A^{-1}\A\U\U^{-1}=\bigcap\U^{\pm(4n)}=\Delta_X$.
Applying Theorem~\ref{t2.2} to the pre-uniformity $\A=\U^{\vee(2n)}$, we get
$$iw(X)\le\overline{\psi}(\U^{\pm(4n-1)})\cdot\ell(\U^{\pm(2n)}\vee\U^{\mp(2n)})\le \chi(\U)\cdot q\ell^{\vee(2n)}(X).$$
\smallskip

4. If $\U$ is $\mp(4n+1)$-separated, then for the pre-uniformity $\A=\U^{\mp(2n)}$ we get
$$\A^{-1}\A\U=\U^{\mp(2n)}\U^{\mp(2n)}\U= \U^{\mp(4n)}$$ and hence $\bigcap\overline{\A^{-1}\A\U}\subset \bigcap \A^{-1}\A\U\U^{-1}=\bigcap\U^{\mp(4n+1)}=\Delta_X$.
Applying Theorem~\ref{t2.2} to the pre-uniformity $\A=\U^{\mp(2n)}$, we get
$$iw(X)\le\overline{\psi}(\U^{\mp(4n)})\cdot\ell(\U^{\mp(2n)})\le \chi(\U)\cdot q\ell^{\mp(2n)}(X).$$
\end{proof}

Corollary~\ref{c2.3} implies:

\begin{corollary}\label{c2.4} If $X$ is a Hausdorff space and $\U$ is a normal quasi-uniformity generating the topology of $X$, then the space $X$ has $i$-weight  $iw(X)\le\overline{\psi}(\U)\cdot\ell(\U\vee  \U^{-1})\le\chi(\U)\cdot\ell(\U^{\vee1})$. Moreover, if the quasi-uniformity $\U$ is
\begin{enumerate}
\item $\mp3$-separated, then $iw(X)\le\overline{\psi}(\U^{\mp2})\cdot\ell(\U)\le \chi(\U)\cdot q\ell^{\pm1}(X)$;
\item $\pm4$-separated, then $iw(X)\le\overline{\psi}(\U^{\pm3})\cdot\ell(\U^{\vee2})\le \chi(\U)\cdot q\ell^{\vee2}(X)$;
\item $\mp5$-separated, then $iw(X)\le\overline{\psi}(\U^{\mp4})\cdot\ell(\U^{\mp2})\le \chi(\U)\cdot q\ell^{\mp2}(X)$;
\item $\pm6$-separated, then  $iw(X)\le\overline{\psi}(\U^{\pm5})\cdot\ell(\U^{\vee3})\le \chi(\U)\cdot q\ell^{\vee3}(X)$;
\item $\mp7$-separated, then $iw(X)\le\overline{\psi}(\U^{\mp6})\cdot\ell(\U^{\pm3})\le \chi(\U)\cdot q\ell^{\pm3}(X)$;
\item $\pm 8$-separated, then $iw(X)\le\overline{\psi}(\U^{\pm7})\cdot\ell(\U^{\vee4})\le \chi(\U)\cdot q\ell^{\vee4}(X)$;
\item $\mp9$-separated, then $iw(X)\le\overline{\psi}(\U^{\mp8})\cdot\ell(\U^{\mp4})\le \chi(\U)\cdot q\ell^{\mp4}(X)$;
\item $\pm10$-separated, then  $iw(X)\le\overline{\psi}(\U^{\pm9})\cdot\ell(\U^{\vee5})\le \chi(\U)\cdot \dc(X)$.
\end{enumerate}
\end{corollary}

\section{Bi-quasi-uniformizable spaces}\label{s3}

In this section we introduce so-called bi-quasi-uniformizable spaces and obtain some upper bounds on the submetrizability number and $i$-weight of such spaces. As a motivation, consider the following characterization.

\begin{proposition}\label{p3.1} For two quasi-uniformities $\mathcal L$ and $\mathcal R$ on a set $X$ the following conditions are equivalent:
\begin{enumerate}
\item $\mathcal L\mathcal R^{-1}\subset\mathcal R^{-1}\mathcal L$;
\item $\mathcal R\mathcal L^{-1}\subset\mathcal L^{-1}\mathcal R$;
\item $\mathcal L\mathcal R^{-1}$ is a quasi-uniformity;
\item $\mathcal R\mathcal L^{-1}$ is a quasi-uniformity.
\end{enumerate}
\end{proposition}

\begin{proof} $(1)\Leftrightarrow(2)$ and $(3)\Leftrightarrow (4)$: Since $\big(\mathcal L\mathcal R^{-1}\big)^{-1}= \mathcal R\mathcal L^{-1}$, the inclusion $\mathcal L\mathcal R^{-1}\subset\mathcal R^{-1}\mathcal L$ is equivalent to $\mathcal R\mathcal L^{-1}\subset\mathcal L^{-1}\mathcal R$.
By the same reason, $\mathcal L\mathcal R^{-1}$ is a quasi-uniformity if and only if $\mathcal R\mathcal L^{-1}$ is a quasi-uniformity.
\smallskip

$(1)\Ra(3)$:  If  $\mathcal L\mathcal R^{-1}\subset\mathcal R^{-1}\mathcal L$, then
$$\mathcal L\mathcal R^{-1}=(\mathcal L\mathcal L)(\mathcal R^{-1}\mathcal R^{-1})=\mathcal L(\mathcal L\mathcal R^{-1})\mathcal R^{-1}\subset \mathcal L(\mathcal R^{-1}\mathcal L)\mathcal R^{-1}=(\mathcal L\mathcal R^{-1})(\mathcal L\mathcal R^{-1}),$$which means that
the pre-uniformity $\mathcal L\mathcal R^{-1}$ is a quasi-uniformity.

$(3)\Ra(1)$: If $\mathcal L\mathcal R^{-1}$ is a quasi-uniformity, then $\mathcal L\mathcal R^{-1}=\mathcal L\mathcal R^{-1}\mathcal L\mathcal R^{-1}\subset \mathcal R^{-1}\mathcal L$.
\end{proof}

Motivated by Proposition~\ref{p3.1} let us introduce the following

\begin{definition} Two quasi-uniformities $\mathcal L$ and $\mathcal R$ on a set $X$ are called
\begin{itemize}
\item {\em commuting} if $\mathcal L\mathcal R=\mathcal R\mathcal L$;
\item {\em $\pm$-subcommuting} if $\mathcal L\mathcal R^{-1}\subset \mathcal R^{-1}\mathcal L$ and $\mathcal R\mathcal L^{-1}\subset \mathcal L^{-1}\mathcal R$;
\end{itemize}
A topological space $X$ is defined to be {\em bi-quasi-uniformizable} if the topology of $X$ is generated by two $\pm$-subcommuting quasi-uniformities.
\end{definition}

\begin{theorem}\label{t3.3} For any $\pm$-subcommuting quasi-uniformities $\mathcal L,\mathcal R$ generating the topology $\tau$ of a topological space $X$ the pre-uniformity $\FU=\mathcal L\mathcal R^{-1}\vee\mathcal R\mathcal L^{-1}$ is a uniformity generating a completely regular topology $\tau_\FU$, weaker than the topology $\tau$ of $X$. If the space $X$ is Hausdorff, then the topology $\tau_\FU$ generated by the uniformity $\FU$ is Tychonoff, the space $X$ is functionally Hausdorff and has submetrizability number
$$sm(X)\le\psi(\FU)\le\chi(\mathcal L)\cdot\chi(\mathcal R)$$ and
$i$-weight
$$iw(X)\le \psi(\FU)\cdot\log(\ell(\FU))\le \chi(\mathcal L)\cdot\chi(\mathcal R)\cdot \log(\dc(X)).$$
\end{theorem}

\begin{proof} By Proposition~\ref{p3.1}, the pre-uniformity $\FU$ is a quasi-uniformity. Since $\FU^{-1}=\FU$, it is a uniformity. Then the topology $\tau_\FU$ generated by the uniformity $\FU$ is Tychonoff (see \cite[8.1.13]{Eng}) Since $\FU\subset\mathcal L$, the topology $\tau_\FU$ is weaker than the topology $\tau_{\mathcal L}=\tau$ of the space $X$.

Now assume that the topology $\tau$ is Hausdorff. In this case for any distinct points $x,y\in X$ we can find entourages $L\in\mathcal L$ and $R\in\mathcal R$ such that $B(x;L)\cap B(y;R)=\emptyset$. Then $y\notin B(x;LR^{-1})$ and hence $(y,x)\notin \bigcap\FU$, which means that the uniformity $\FU$ is separated and the topology $\tau_\FU$ generated by $\FU$ is Tychonoff. Consequently, the space $X$ is functionally Hausdorff.

To show that $sm(X)\le\psi(\FU)$, fix a subfamily $\V\subset\FU$ of cardinality $|\V|=\psi(\FU)$ such that $\bigcap\V=\Delta_X$. By \cite[8.1.11]{Eng}, for every entourage $V\in\V$ there exists a continuous pseudometric $d_V$ on $X$ such that the entourage $[d_V]_{<1}=\{(x,y)\in X\times X:d_V(x,y)<1\}$ is contained in $V$. Then the family of pseudometrics $\mathcal D=\{d_V\}_{V\in\V}$ separates points of $X$, which implies that $sm(X)\le |\mathcal D|\le|\V|=\psi(\FU)$.

Taking into account that the topological weight of a metric space is equal to its boundedness number, which does not exceed the discrete cellularity, and applying Proposition~\ref{p1.3}, we conclude that $$iw(X)\le \psi(\FU)\cdot\log(\ell(\FU))\le \chi(\FU)\cdot \log(\dc(X))\le\chi(\mathcal L)\cdot\chi(\mathcal R)\cdot \log(\dc(X)).$$
\end{proof}

Theorem~\ref{t3.3} implies:

\begin{corollary}\label{c3.4} Each Hausdorff bi-quasi-uniformizable topological space is functionally Hausdorff.
\end{corollary}

We do not know if this corollary can be reversed.

\begin{problem} Is each functionally Hausdorff space bi-quasi-uniformizable?
\end{problem}

\begin{proposition}\label{p3.5} Let $\mathcal L,\mathcal R$ be two $\pm$-subcommuting quasi-uniformities generating the same Hausdorff topology on $X$. If the quasi-uniformities $\mathcal L^{-1},\mathcal R^{-1}$ generate the same topology on $X$, then the quasi-uniformities $\mathcal L$ and $\mathcal R$ are 3-separated. \end{proposition}

\begin{proof} Given two distinct points $x,y\in X$ we shall find an entourage $R\in\mathcal R$ such that  $(x,y)\notin R^{-1}RR^{-1}$. Since the topology generated by the quasi-uniformities $\mathcal L$ and $\mathcal R$ on $X$ is Hausdorff, there are two entourages $L\in\mathcal L$ and $\hat R\in\mathcal R$ such that $B(x;\hat R)\cap B(y;LL)=\emptyset$ and hence $(x,y)\notin \hat RL^{-1}L^{-1}$. Replacing $\hat R$ by a smaller entourage, we can additionally assume that $B(y;\hat R)\subset B(y;L)$. Then $B(x;\hat R)\cap B(y;\hat RL)=\emptyset$ and hence $y\notin B(x;\hat RL^{-1}\hat R^{-1})$.
Since the quasi-uniformities $\mathcal L$ and $\mathcal R$ are $\pm$-subcommuting, for the entourages $L$ and $\hat R$ there are entourages $\tilde L\in\mathcal L$ and $\tilde R\in\mathcal R$ such that $\tilde L^{-1}\tilde R\subset \hat RL^{-1}$. Since quasi-uniformities $\mathcal L^{-1}$ and $\mathcal R^{-1}$ generate the same topology on $X$, for the entourage $\tilde L^{-1}$ there is an entourage $\check R\in\mathcal R$ such that $B(x;\check R^{-1})\subset B(x;\tilde L^{-1})$. Then for the entourage $R=\check R\cap \tilde R\cap \hat R$ we get  $B(x;R^{-1}RR^{-1})\subset B(x;\check R^{-1}\tilde R\hat R^{-1})\subset B(x;\tilde L^{-1}\tilde R\hat R^{-1})\subset B(x;\hat RL^{-1}\hat R^{-1})$ and hence $y\notin B(x;R^{-1}RR^{-1})$. So, $\bigcap\mathcal R^{-1}\mathcal R\mathcal R^{-1}=\Delta_X$ and after inversion, $\bigcap\mathcal R\mathcal R^{-1}\mathcal R=\Delta_X$, which means that the quasi-uniformity $\mathcal R$ is 3-separated. By analogy we can prove that the quasi-uniformity $\mathcal L$ is 3-separated.
\end{proof}

\section{Normally bi-quasi-uniformizable spaces}\label{s4}

Observe that for two quasi-uniformities $\mathcal L,\mathcal R$ on a set $X$ the inclusion $\mathcal L\mathcal R^{-1}\subset \mathcal R^{-1}\mathcal L$ is equivalent to the existence for every entourages $L\in\mathcal L$ and $R\in\mathcal R$ two  entourages $\tilde L\in\mathcal L$ and $\tilde R\in\mathcal R$ such that $\tilde R^{-1}\tilde L\subset LR^{-1}$. Changing the order of quantifiers in this property we obtain the following notion.

\begin{definition} A topological space $X$ is called {\em normally bi-quasi-uniformizable} if its topology is generated by quasi-uniformities $\mathcal L$ and $\mathcal R$ satisfying the following  properties:
\begin{itemize}
\item $\forall L\in\mathcal L\;\;\exists \tilde L\in\mathcal L\;\;\forall R\in\mathcal R\;\;\exists \tilde R\in\mathcal R$ such that $\tilde R^{-1}\tilde L\subset LR^{-1}$ and $\tilde L^{-1}\tilde R\subset RL^{-1}$;
\item $\forall R\in\mathcal R\;\;\exists \tilde R\in\mathcal R\;\;\forall L\in\mathcal L\;\;\exists \tilde L\in\mathcal L$ such that $\tilde L^{-1}\tilde R\subset RL^{-1}$ and $\tilde R^{-1}\tilde L\subset RL^{-1}$.
\end{itemize}
In this case we shall say that the quasi-uniformities $\mathcal L$ and $\mathcal R$ are {\em normally $\pm$-subcommuting}.
\end{definition}

By analogy we can introduce normally commuting quasi-uniformities.

\begin{definition} Two quasi-uniformities $\mathcal L$ and $\mathcal R$ on a set $X$ are defined to be
{\em normally commuting} if it satisfy the following two conditions:
\begin{itemize}
\item $\forall L\in\mathcal L\;\;\exists \tilde L\in\mathcal L\;\;\forall R\in\mathcal R\;\;\exists \tilde R\in\mathcal R$ such that $\tilde R\tilde L\subset LR$ and $\tilde L\tilde R\subset RL$;
\item $\forall R\in\mathcal R\;\;\exists \tilde R\in\mathcal R\;\;\forall L\in\mathcal L\;\;\exists \tilde L\in\mathcal L$ such that $\tilde L\tilde R\subset RL$ and $\tilde R\tilde L\subset RL$.
\end{itemize}
\end{definition}

\begin{proposition}\label{p4.3} Any two normally $\pm$-subcommuting quasi-uniformities $\mathcal L,\mathcal R$ generating the same  topology on a set $X$ are normal. Consequently, each normally bi-quasi-uniformizable topological space is normally quasi-uniformizable.
\end{proposition}

\begin{proof} To show that $\mathcal L$ is normal, fix a subset $A\subset X$ and entourage $L\in\mathcal L$. Since $\mathcal L$ and $\mathcal R$ are normally $\pm$-subcommuting, for the entourage $L$ there exists an entourage $\tilde L\in\mathcal L$ such that for every entourage $R\in\mathcal R$ there is an entourage $\tilde R\in\mathcal R$ with $\tilde L^{-1}\tilde R\subset RL^{-1}$. We claim that $B(\bar{A};\tilde L)\subset\overline{B(A;L)}$. Given any point $x\in B(\bar{A};\tilde L)$, we need to show that $x\in \overline{B(A;L)}$. Given any neighborhood $O_x\subset X$ of $x$, find an entourage $R\in\mathcal R$ such that $B(x;R)\subset O_x$. By the choice of the entourage $\tilde L$, for the entourage $R$ there is an entourage $\tilde R\in\mathcal R$ such that $\tilde L^{-1}\tilde R\subset RL^{-1}$. It follows from $x\in B(\bar{A};\tilde L)$ that $B(x;\tilde L^{-1})\cap \bar{A}\ne\emptyset$ and hence $\emptyset \ne B(x;\tilde L^{-1}\tilde R)\cap A\subset B(x;RL^{-1})\cap A$. Then $\emptyset\ne B(x;R)\cap B(A;L)\subset O_x\cap B(A;L)$, which means $x\in \overline{B(A;L)}$. So, $B(\bar{A};\tilde L)\subset \overline{B(A;L)}$ and hence $\bar{A}\subset B(\bar{A};\tilde L)^\circ\subset\overline{B(A;L)}^\circ$, which means that $\mathcal L$ is normal. By analogy we can prove the normality of the quasi-uniformity $\mathcal R$.
\end{proof}

\begin{theorem}\label{t4.4} If $\mathcal L$ and $\mathcal R$ are two normally $\pm$-subcommuting quasi-uniformities generating the topology of a Hausdorff topological space $X$, then the quasi-uniformities $\mathcal L\mathcal R^{-1}$ and $\mathcal R\mathcal L^{-1}$ are 1-separated and have pseudocharacter
\begin{enumerate}
\item $\psi(\mathcal L\mathcal R^{-1})=\psi(\mathcal R\mathcal L^{-1})\le \psi(\mathcal L\mathcal L^{-1})\cdot\ell(\mathcal L^{-1})\le \psi(\mathcal L\mathcal L^{-1})\cdot q\ell^{\mp1}(X)$;
\item $\psi(\mathcal L\mathcal R^{-1})=\psi(\mathcal R\mathcal L^{-1})\le\psi(\mathcal L^{-1}\mathcal L)\cdot \ell(\mathcal L)\le\psi(\mathcal L^{-1}\mathcal L)\cdot q\ell^{\pm1}(X)$ if $\mathcal L^{-1},\mathcal R^{-1}$ are normally $\pm$-sub\-commuting and generate the same topology on $X$;
\item   $\psi(\mathcal L\mathcal R^{-1})=\psi(\mathcal R\mathcal L^{-1})\le\psi(\mathcal L\mathcal L^{-1}\mathcal L)\cdot\ell(\mathcal L\mathcal L^{-1}\vee\mathcal L^{-1}\mathcal L)\le\psi(\mathcal L\mathcal L^{-1}\mathcal L)\cdot q\ell^{\vee 2}(X)$ if the quasi-uniformities $\mathcal L$ and $\mathcal R$ are normally commuting and $\bigcap\mathcal L\mathcal L^{-1}\mathcal L=\Delta_X$;
\item $\psi(\mathcal L\mathcal R^{-1})=\psi(\mathcal R\mathcal L^{-1})\le \overline{\psi}(\A^{-1}\A\mathcal L)\cdot\ell(\A)\cdot \ell^{\pm2}(X)$ for any pre-uniformity $\A$ on $X$ such that\newline  $\bigcap\overline{\A^{-1}\kern-2pt\A\mathcal L}=\Delta_X$.
\end{enumerate}
\end{theorem}

\begin{proof}  First we show that the quasi-uniformities $\mathcal L\mathcal R^{-1}$ and $\mathcal R\mathcal L^{-1}$ are 1-separated. Since the topology of $X$ is Hausdorff, for any distinct points $x,y\in X$ we can find two disjoint open sets $O_x\ni x$ and $O_y\ni y$. Taking into account that the quasi-uniformities $\mathcal L$ and $\mathcal R$ generate the topology of $X$, we can find two entourages $L\in\mathcal L$ and $R\in\mathcal R$ such that $B(x;L)\subset O_x$ and $B(y;R)\subset O_y$. Then $B(x;L)\cap B(y;R)=\emptyset$ and hence $y\notin B(x;LR^{-1})$ and $x\notin B(y;RL^{-1})$, which implies that $\bigcap\mathcal L\mathcal R^{-1}=\Delta_X=\bigcap\mathcal R\mathcal L^{-1}$. So, the quasi-uniformities $\mathcal L\mathcal R^{-1}$ and $\mathcal R\mathcal L^{-1}$ are 1-separated. Taking into account that $\big(\mathcal L\mathcal R^{-1}\big)^{-1}=\mathcal R\mathcal L^{-1}$ we conclude that $\psi(\mathcal L\mathcal R^{-1})=\psi(\mathcal R\mathcal L^{-1})$.
\smallskip

1. Now we shall prove the inequality $\psi(\mathcal L\mathcal R^{-1})\le \psi(\mathcal L\mathcal L^{-1})\cdot\ell(\mathcal L^{-1})$. Fix a family of entourages $\Lambda\subset\mathcal L$ of cardinality $|\Lambda|\le\psi(\mathcal L\mathcal L^{-1})$ such that $\bigcap_{L\in\Lambda}LL^{-1}=\Delta_X$. Replacing every $L\in\Lambda$ by a smaller entourage, we can assume that $\bigcap_{L\in\Lambda}(LL)(LL)^{-1}=\Delta_X$.

Since the quasi-uniformities $\mathcal L$ and $\mathcal R$ are normally $\pm$-subcommuting, for the entourage $L\in\mathcal L$ there exists an entourage $\tilde L\in\mathcal L$ such that for any entourage $R\in\mathcal R$ there exists an entourage $\tilde R\in\mathcal R$ such that $\tilde L^{-1}\tilde R\subset RL^{-1}$. Replacing $\tilde L$ by $\tilde L\cap L$, we can assume that $\tilde L\subset L$.
For the entourage $\tilde L$ choose a subset $Z_L\subset X$ of cardinality $|Z_L|\le \ell(\mathcal L^{-1})$ such that  $X=B(Z_L;\tilde L^{-1})$. For every $z\in Z_L$ choose an entourage $R_z\in\mathcal R$ such that $B(z;R_z)\subset B(z; L)$. By the choice of $\tilde L$, for the entourage $R_z$ there exists an entourage $\tilde R_z\in\mathcal R$ such that $\tilde L^{-1}\tilde R_z\subset R_zL^{-1}$.
Consider the family
$$\mathcal P=\bigcup_{L\in\Lambda}\{(L,\tilde R_z):z\in Z_L\}\subset\mathcal L\times\mathcal R.$$ We claim that for any distinct points $x,y\in X$ there is a pair $(L,\tilde R_z)\in\mathcal P$ such that $B(x;L)\cap B(y;\tilde R_z)=\emptyset$. By the choice of the family $\Lambda$, there is an entourage $L\in\Lambda$ such that $x\notin B(y;LLL^{-1}L^{-1})$.  Since $y\in X=B(Z_L;\tilde L^{-1})$, there exists a point $z\in Z_L$ such that $y\in B(z;\tilde L^{-1})$ and hence $z\in B(y;\tilde L)$. We claim that the pair $(\tilde L,\tilde R_z)\in\mathcal P$ has the desired property: $B(x; L)\cap B(y;\tilde R_z)=\emptyset$. Assuming that $B(x;L)\cap B(y;\tilde R_z)\ne\emptyset$, we would conclude that $$x\in B(y;\tilde R_zL^{-1}){\subset} B(z;\tilde L^{-1}\tilde R_z L^{-1}){\subset} B(z;R_zL^{-1}L^{-1}){\subset} B(z;LL^{-1}L^{-1}){\subset} B(y;\tilde LLL^{-1}L^{-1}){\subset} B(y,LLL^{-1}L^{-1})$$
which contradicts the choice of $L$. So $B(x;L)\cap B(y;\tilde R_z)=\emptyset$, which is equivalent to
$y\notin B(x;L\tilde R_z^{-1})$. Then $$\psi(\mathcal L\mathcal R^{-1})\le |\mathcal P|\le\sum_{L\in\Lambda}|Z_L|\le |\Lambda|\cdot \ell(\mathcal L^{-1})\le \psi(\mathcal L\mathcal L^{-1})\cdot \ell(\mathcal L^{-1}).$$
\smallskip

2. If the quasi-uniformities $\mathcal L^{-1}$ and $\mathcal R^{-1}$ are normally $\pm$-subcommuting and generate the same topology on $X$, then by Proposition~\ref{p3.5}, this topology is Hausdorff, which allows us to apply the first item to the quasi-uniformities $\mathcal L^{-1},\mathcal R^{-1}$ and obtain the upper bound
$\psi(\mathcal L^{-1}\mathcal R)\le\psi(\mathcal L^{-1}\mathcal L)\cdot \ell(\mathcal L)$.
The $\pm$-subcommutativity of $\mathcal L^{-1}$ and $\mathcal R^{-1}$ implies that $\psi(\mathcal R\mathcal L^{-1})\le\psi(\mathcal L^{-1}\mathcal R)$. So,
$$\psi(\mathcal L\mathcal R^{-1})=\psi(\mathcal R\mathcal L^{-1})\le\psi(\mathcal L^{-1}\mathcal R)\le\psi(\mathcal L^{-1}\mathcal L)\cdot \ell(\mathcal L)\le\psi(\mathcal L^{-1}\mathcal L)\cdot q\ell^{\pm1}(X).$$
\smallskip

3. Next, assuming that the quasi-uniformities $\mathcal L$ and $\mathcal R$ are normally commuting and $\bigcap\mathcal L\mathcal L^{-1}\mathcal L=\Delta_X$, we prove the inequality $\psi(\mathcal R\mathcal L^{-1})=\psi(\mathcal L\mathcal R^{-1})\le
\psi(\mathcal L\mathcal L^{-1}\mathcal L)\cdot\ell(\mathcal L\mathcal L^{-1}\vee\mathcal L^{-1}\mathcal L)$. Fix a subfamily $\Lambda\subset\mathcal L$ of cardinality $|\Lambda|=\psi(\mathcal L\mathcal L^{-1}\mathcal L)$ such that $\bigcap_{L\in\Lambda}LL^{-1}L=\Delta_X$. Replacing every entourage $L\in\Lambda$ by a smaller entourage, we can assume that $\bigcap_{L\in\Lambda}L^{2}L^{-3}L=\Delta_X$.

Since the quasi-uniformities $\mathcal L$ and $\mathcal R$ are normally commuting and normally $\pm$-subcommuting, for every entourage $L\in\Lambda$ there exists an entourage $\tilde L\in\mathcal L$, $\tilde L\subset L$, such that for every entourage $R\in\mathcal R$ there exists an entourage $\tilde R\in\mathcal R$ such that $\tilde L\tilde R\subset RL$ and $\tilde L^{-1}\tilde R\subset RL^{-1}$.

By the definition of the boundedness number $\ell(\mathcal L\mathcal L^{-1}\kern-2pt\vee\kern-2pt\mathcal L^{-1}\kern-2pt\mathcal L)$, for every $L\in\Lambda$ there exists a subset $A_L\subset X$ of cardinality $|A_L|\le \ell(\mathcal L\mathcal L^{-1}\kern-2pt\vee\kern-2pt\mathcal L^{-1}\kern-2pt\mathcal L)$ such that $X=B(A_L;\tilde L\tilde L^{-1}\cap \tilde L^{-1}\tilde L)$.

For every point $a\in A_L$ choose an entourage $R_a\in\mathcal R$ such that $B(a;R_a)\subset B(a;L)$. By the choice of $\tilde L$ for the entourage $R_a$ there exists an entourage $\check R_a\in\mathcal L$ such that $\tilde L\check R_a\subset R_aL$, and for the entourage $\check R_a\in\mathcal R$ there is an entourage $\tilde R_a\in\mathcal R$ such that $\tilde L^{-1}\tilde R_a\subset \check R_aL^{-1}$.
Consider the family of pairs
$$\mathcal P=\bigcup_{L\in\Lambda}\{(L,\tilde R_a):a\in A_L\}\subset\mathcal L\times\mathcal R.$$
We claim that for any distinct points $x,y\in X$ there exists a pair $(L,R)\in\mathcal P$ such that $B(x;L)\cap B(y;R)=\emptyset$. Given two distinct points $x,y\in X$, find an entourage $L\in\Lambda$ such that $(x,y)\notin L^2L^{-3}L$.

 Since $y\in X=B(A_L;\tilde L\tilde L^{-1}\cap \tilde L^{-1}\tilde L)$, we can find a point $a\in A_L$ such that $y\in B(a;\tilde L\tilde L^{-1}\cap \tilde L^{-1}\tilde L)$ and hence $y\in B(a;\tilde L\tilde L^{-1})$ and $a\in B(y;\tilde L^{-1}\tilde L)\subset B(y;L^{-1}L)$. We claim that $B(x;L)\cap B(y;\tilde R_a)=\emptyset$. To derive a contradiction, assume that
$B(x;L)\cap B(y;\tilde R_a)\ne\emptyset$. Observe that $$B(y;\tilde R_a)\subset B(a;\tilde L\tilde L^{-1}\tilde R_a)\subset B(a;\tilde L\check R_aL^{-1})\subset B(a;R_aLL^{-1})\subset B(a;LLL^{-1})\subset B(y;L^{-1}LLLL^{-1}).$$ Then
$\emptyset\ne B(x;L)\cap B(y;\tilde R_a)\subset B(x;L)\cap B(y;L^{-1}LLLL^{-1})$ implies $y\notin B(x;L^2L^{-3}L)$, which contradicts the choice of the entourage $L$. This contradiction shows that
$B(x;L)\cap B(y;\tilde R_a)=\emptyset$ and hence $$\psi(\mathcal R\mathcal L^{-1})=\psi(\mathcal L\mathcal R^{-1})\le|\mathcal P|\le\sum_{L\in\Lambda}|A_V|\le \psi(\mathcal L\mathcal L^{-1}\mathcal L)\cdot \ell(\mathcal L\mathcal L^{-1}\kern-2pt\vee\kern-2pt\mathcal L^{-1}\kern-2pt\mathcal L).$$
\smallskip

4. Finally we prove that  $\psi(\mathcal L\mathcal R^{-1})=\psi(\mathcal R\mathcal L^{-1})\le \overline{\psi}(\A^{-1}\A\mathcal L)\cdot\ell(\A)\cdot \ell^{\pm2}(X)$ for any pre-uniformity $\A$ on $X$ such that $\bigcap\overline{\A^{-1}\A\U}=\Delta_X$. If $\overline{\psi}(\A^{-1}\A\mathcal L)$ is finite, then $\overline{\psi}(\A^{-1}\A\mathcal L)=1$, which implies that $A^{-1}AL=\Delta_X=A=L$ for some $A\in\A$ and $L\in\mathcal L$. In this case  $\ell(\A)=|X|$  and the topological space $X$ is discrete. Then for every point $x\in X$ we can choose an entourage $R_x\in\mathcal R$ such that $B(x;R_x)=\{x\}$. Then $\bigcap_{x\in X}R_xL^{-1}=\bigcap_{x\in X}R_x=\Delta_X$ and hence $\psi(\mathcal R\mathcal L^{-1})\le |X|=\ell(\A)\le \overline{\psi}(\A^{-1}\A\mathcal L)\cdot\ell(\A)\cdot \ell^{\pm2}(X)$.

So, we assume that the cardinal $\kappa=\overline{\psi}(\A^{-1}\A\U)$ is infinite. Since $\bigcap\overline{\A^{-1}\A\mathcal L}=\Delta_X$, we can choose subfamilies $(A_\alpha)_{\alpha\in\kappa}\subset\A$ and $(L_\alpha)_{\alpha\in\kappa}\subset\mathcal L$ such that $\bigcap_{\alpha<\kappa}\overline{B(x,A_\alpha^{-1}A_\alpha L^3_\alpha)}=\{x\}$ for every $x\in X$.

For every $\alpha<\kappa$ consider the entourage $A_\alpha\in\A$ and find a subset $Z_\alpha\subset X$ of cardinality $|Z_\alpha|\le\ell(\A)$ such that $X=B(Z_\alpha;A_\alpha)$.
Since the quasi-uniformities $\mathcal L$ and $\mathcal R$ are normally $\pm$-subcommuting, for the entourage $L_\alpha$ there is an entourage $\tilde L_\alpha$ such that for every $R\in\mathcal R$ there is $\tilde R_\alpha\in\mathcal R$ such that $\tilde L_\alpha^{-1}\tilde R\subset RL_\alpha^{-1}$.

Now fix any point $z\in Z_\alpha$.
The normality of the quasi-uniformity $\mathcal L$ (proved in Proposition~\ref{p4.3}) guarantees that $\overline{B(z;A_\alpha L_\alpha^2)}\subset  \overline{B(z;A_\alpha L_\alpha^3)}^\circ$. Put $W_{\alpha,z}=\overline{B(z;A_\alpha L_\alpha^3)}^\circ$.
For every point $y\in X\setminus W_{\alpha,z}$ choose an entourage $R_y\in\mathcal R$ such that $B(y;R_yR_y)\cap \overline{B(z;A_\alpha L_\alpha^2)}=\emptyset$ and hence $B(y;R^2_yL_\alpha^{-1})\cap B(z;A_\alpha L_\alpha)=\emptyset$. For every $y\in X\setminus \overline{B(z;A_\alpha L_\alpha^3)}$ we can replace $R_y$ by a smaller entourage and assume additionally that $B(y;R_y)$ is disjoint with $\overline{B(z;A_\alpha L_\alpha^3)}$.

By the choice of the entourage $\tilde L_\alpha$  for every $y\in X\setminus W_{\alpha,z}$ there is an entourage $\tilde R_y\in\mathcal R$ such that $\tilde R_y\subset R_y$ and $\tilde L_\alpha^{-1}\tilde R_y\subset R_yL^{-1}_\alpha$. For every $y\in W_{\alpha,z}$ choose an entourage $\tilde R_y\in\mathcal R$ such that $B(y;\tilde R_y)\subset W_{\alpha,z}$. Now consider the neighborhood assignment $V=\bigcup_{y\in X}\{y\}\times B(y;\tilde R_y\cap \tilde L_\alpha)$. By the definition of $\ell^{\pm2}(X)$, there exists a subset $A_{\alpha,z}\subset X$ of cardinality $|A_{\alpha,z}|\le \ell^{\pm2}(X)$ such that $X=B(A_{\alpha,z};VV^{-1})$.

%Then for every $y\notin\overline{W}$ the ball $B(y;\tilde L_\alpha)$ meets some ball $B(a;\tilde R_a)$ with $a\in A_{\alpha,z}$. Then $B(y;\tilde R_a)\subset B(a;\tilde R_a\tilde L_\alpha^{-1}\tilde R_a)\subset B(a;\tilde R_aR_aL_\alpha^{-1})\subset B(a;R_a^2L_\alpha^{-1})$. Taking into account that the set $B(a;R_a^2L_\alpha^{-1})$ is disjoint with the set $B(z;A_\alpha L_\alpha)$, we conclude that $B(y;\tilde R_a)\cap B(z;A_\alpha L_\alpha)=\emptyset$.

Consider the family $\mathcal P=\bigcup_{\alpha\in \kappa}\bigcup_{z\in Z_\alpha}\{(L_\alpha,\tilde R_a):a\in A_{\alpha,z}\}\subset\mathcal L\times\mathcal R$. We claim that for any distinct points $x,y\in X$ there is a pair $(L,R)\in\mathcal P$ such that $B(x;L)\cap B(y;R)=\emptyset$.

Indeed, for the points $x,y\in X$ we can find an ordinal $\alpha\in\kappa$ such that $y\notin \overline{B(x;A_\alpha^{-1}\A_\alpha L_\alpha^3)}$.
Since $X=B(Z_\alpha;A_\alpha)$, there is a point $z\in Z_\alpha$ such that $x\in B(z;A_\alpha)$.
Then $y\notin \overline{B(z;A_\alpha L_\alpha^3)}$ and hence $B(y,\tilde R_y)\subset B(y;R_y)$ is disjoint with $\overline{B(z;A_\alpha L_\alpha^3)}$ by the choice of the entourage $R_y$.

Since $y\in X=B(A_{\alpha,z};VV^{-1})$, there is a point $a\in A_{\alpha,z}$ such that $y\in B(a;VV^{-1})$, which implies that $\emptyset\ne B(y;V)\cap B(a;V)=B(y;\tilde R_y\cap \tilde L_\alpha)\cap B(a;\tilde R_a\cap \tilde L_\alpha)$ and
hence $y\in B(a;\tilde R_a\tilde L_\alpha^{-1})$. Since $B(y,\tilde R_y)$ is disjoint with $W_{\alpha,z}$, the choice of the entourage $R_a$ guarantees that $a\notin W_{\alpha,z}$ and hence $B(a;R_aR_a)\cap\overline{B(z;A_\alpha L_\alpha^2)}=\emptyset$ and $B(a;R_aR_aL_\alpha^{-1})\cap B(z;A_\alpha L_\alpha)=\emptyset$.
Now observe that the $\tilde R_a$-ball  $B(y;\tilde R_a)\subset B(a;VV^{-1}\tilde R_a)\subset B(a;R_a\tilde L_\alpha^{-1}\tilde R_a)\subset B(a;R_aR_aL_\alpha^{-1})$ is disjoint with the $L_\alpha$-ball $B(x;L_\alpha)\subset B(z;A_\alpha L_\alpha)$.

The family $\mathcal P$ witnesses that
$$\psi(\mathcal L\mathcal R^{-1})=\psi(\mathcal R\mathcal L^{-1})\le|\mathcal P|\le\overline{\psi}(\A^{-1}\A\mathcal L)\cdot \ell(\A)\cdot \ell^{\pm2}(X).$$
\end{proof}

Taking into account that $\psi(\mathcal L\mathcal R^{-1}\vee \mathcal R\mathcal L^{-1})\le\psi(\mathcal L\mathcal R^{-1})$, and applying Theorem~\ref{t4.4}  we obtain:

\begin{theorem}\label{t4.5} Let $X$ be a Hausdorff  topological space and $\mathcal L,\mathcal R$ be two normally $\pm$-subcommuting quasi-uniformities generating the topology of $X$. Then the uniformity $\FU=\mathcal L\mathcal R^{-1}\vee\mathcal R\mathcal L^{-1}$ has pseudocharacter:
\begin{enumerate}
\item $\psi(\FU)\le\overline{\psi}(\mathcal L)\cdot\ell(\mathcal L\vee \mathcal L^{-1})\cdot \ell^{\pm2}(X)$;
\item $\psi(\FU)\le \psi(\mathcal L\mathcal L^{-1})\cdot \ell(\mathcal L^{-1})\le\psi(\mathcal L^{\pm2})\cdot q\ell^{\mp1}(X)$.
\end{enumerate}
Moreover, if the quasi-uniformity $\mathcal L$ is
\begin{enumerate}
\item[(3)] $\mp3$-separated, then $\psi(\FU)\le\overline{\psi}(\mathcal L^{-1}\mathcal L)\cdot\ell(\mathcal L)\cdot \ell^{\pm2}(X)\le \overline{\psi}(\mathcal L^{\mp2})\cdot \ell^{\pm1}(X)$.
\item[(4)] $\pm 4$-separated, then $\psi(\FU)\le\overline{\psi}(\mathcal L\mathcal L^{-1}\mathcal L)\cdot\ell(\mathcal L\mathcal L^{-1}\vee \mathcal L^{-1}\mathcal L)\cdot \ell^{\pm2}(X)\le
\overline{\psi}(\mathcal L^{\pm3})\cdot\ell^{\vee 2}(X)$;
\item[(5)] $\mp5$-separated, then $\psi(\FU)\le\overline{\psi}(\mathcal L^{-1}\mathcal L\mathcal L^{-1}\mathcal L)\cdot\ell(\mathcal L^{-1}\mathcal L)\cdot \ell^{\pm2}(X)\le \overline{\psi}(\mathcal L^{\mp4})\cdot q\ell^{\mp2}(X)\cdot \ell^{\pm2}(X)$;
\item[(6)] $\pm6$-separated, then
$\psi(\FU)\le\overline{\psi}(\mathcal L\mathcal L^{-1}\mathcal L\mathcal L^{-1}\mathcal L)\cdot \ell^{\pm2}(X)=\overline{\psi}(\mathcal L^{\pm5})\cdot\ell^{\pm2}(X)$.
\end{enumerate}
If the quasi-uniformities $\mathcal L$ and $\mathcal R$ are normally commuting and 3-separated, then
\begin{enumerate}
\item[(7)] $\psi(\FU)\le \psi(\mathcal L\mathcal L^{-1}\mathcal L)\cdot \ell(\mathcal L\mathcal L^{-1}\vee \mathcal L^{-1}\mathcal L)\le\psi(\mathcal L^{\pm3})\cdot q\ell^{\vee2}(X)$.
\end{enumerate}
If the quasi-uniformities $\mathcal L^{-1}$, $\mathcal R^{-1}$ are normally $\pm$-subcommuting and  generate the same topology on $X$, then
\begin{itemize}
\item[(8)] $\psi(\FU)\le\psi(\mathcal L^{-1}\mathcal L)\cdot\ell(\mathcal L)\le\psi(\mathcal L^{\mp2})\cdot q\ell^{\pm1}(X)$ and
\item[(9)] $\psi(\FU)\le\psi(\mathcal L \mathcal L^{-1}\vee \mathcal L^{-1}\mathcal L)\cdot \ell(\mathcal L)\cdot \ell(\mathcal L^{-1})\le \psi(\mathcal L^{\vee2})\cdot q\ell^{\pm1}(X)\cdot q\ell^{\mp1}(X)$.
\end{itemize}
\end{theorem}

\begin{proof} 1. The first inequality follows from Theorem~\ref{t4.4}(4) applied to the pre-uniformity $\A=\U\vee\U^{-1}$.
\smallskip

2. The second item follows from Theorem~\ref{t4.4}(1).
\smallskip

3--6. The items (3)--(6) follow from Theorem~\ref{t4.4}(4) applied to the pre-uniformities $\mathcal L$, $\mathcal L\mathcal L^{-1}\vee\mathcal L^{-1}\mathcal L$, $\mathcal L^{-1}\mathcal L$, and $\mathcal L\mathcal L^{-1}$, respectively.
\smallskip

7. The seventh item follows from Theorem~\ref{t4.4}(3).
\smallskip

8,9. Assume that the quasi-uniformities $\mathcal L^{-1}$, $\mathcal R^{-1}$ are normally $\pm$-subcommuting and generate the same topology on $X$. The inequalities $\psi(\FU)\le\psi(\mathcal L^{-1}\mathcal L)\cdot\ell(\mathcal L)\le\psi(\mathcal L^{\mp2})\cdot q\ell^{\pm1}(X)$ follow from Theorem~\ref{t4.4}(2).

To prove that $\psi(\FU)\le\psi(\mathcal L \mathcal L^{-1}\vee \mathcal L^{-1}\mathcal L)\cdot \ell(\mathcal L)\cdot \ell(\mathcal L^{-1})$, fix a subset $\Lambda\subset \mathcal L$ of cardinality $|\Lambda|=\psi(\mathcal L\mathcal L^{-1}\vee \mathcal L^{-1}\mathcal L)$ such that $\bigcap_{L\in\Lambda}LL^{-1}\cap L^{-1}L=\Delta_X$. Replacing every $L\in\Lambda$ by a smaller entourage, we can assume that $\bigcap_{L\in\Lambda}L^2L^{-2}\cap L^{-2}L^{2}=\Delta_X$.
Since the quasi-uniformities $\mathcal L,\mathcal R$ are normally $\pm$-subcommuting and the quasi-uniformities $\mathcal L^{-1}$, $\mathcal R^{-1}$ are normally $\pm$-subcommuting, for every $L\in\Lambda$ there exists an entourage $\tilde L\in\mathcal L$ with $\tilde L\subset L$ such that for every $R\in\mathcal R$ there is $\tilde R\in\mathcal R$ such that $\tilde L^{-1}\tilde R\subset RL^{-1}$ and $\tilde L\tilde R^{-1}\subset R^{-1}L$.

For every $L\in\Lambda$ fix a subset $Z_L\subset X$ of cardinality $|Z_L|\le\ell(\mathcal L)+\ell(\mathcal L^{-1})$ such that $X=B(Z_L;\tilde L)=B(Z_L;\tilde L^{-1})$.
Since the quasi-uniformities $\mathcal L$, $\mathcal R$ generate the same topology on $X$ and $\mathcal L^{-1}$, $\mathcal R^{-1}$ generate the same topology on $X$, for every $z\in Z_L$ we can choose an entourage $R_z\in\mathcal R$ such that  $B(z;R_z)\subset B(z;L)$ and $B(z;R_z^{-1})\subset B(z;L^{-1})$. By the choice of $\tilde L$ for the entourage $R_z$ there is an entourage $\tilde R_z\in\mathcal R$ such that $\tilde R_z\subset R_z$,  $\tilde L^{-1}\tilde R_z\subset R_zL^{-1}$ and $\tilde L\tilde R_z^{-1}\subset R_z^{-1}L$. For the entourage $\tilde R_z$ there is an entourage $\check R_z\in\mathcal R$ with $\check R_z\subset\tilde R_z$ such that $\tilde L\check R_z^{-1}\subset \tilde R_z^{-1}L$, which is equivalent to $\check R_z\tilde L^{-1}\subset L^{-1}\tilde R_z$.

We claim that the family $\mathcal P=\{(\tilde L,\check R_z):L\in\mathcal L,\;z\in Z_L\}\subset \mathcal L\times\mathcal R$ has $\bigcap_{(L,R)\in\mathcal P}LR^{-1}\cap RL^{-1}=\Delta_X$. Given any distinct points $x,y$ find an entourage $L\in\Lambda$ such that $(x,y)\notin L^2L^{-2}\cap L^{-2}L^2$ and hence $(x,y)\notin  L^2 L^{-2}$ or $(x,y)\notin L^{-2}L^2$.

If $(x,y)\notin L^2L^{-2}$, then $B(y;L^2)\cap B(x;L^2)=\emptyset$. Since $y\in X=B(Z_L;\tilde L^{-1})$, there is $z\in Z_L$ such that $y\in B(z;\tilde L^{-1})\subset B(z;L^{-1})$. Then $z\in B(y;L)$ and the $L$-ball $B(z;L)\subset B(y;LL)$ does not intersect $B(x;L^2)$, which implies $B(z;LL^{-1})\cap B(x;L)=\emptyset$. Observe that $B(y;\tilde R_z)\subset B(z;\tilde L^{-1}\tilde R_z)\subset B(z;R_zL^{-1})\subset B(z;LL^{-1})$ and hence $B(y;\tilde R_z)\cap B(x;L)\subset B(z;LL^{-1})\cap B(x;L)=\emptyset$. So, $(x,y)\notin L\tilde R^{-1}_z$ and hence $(x,y)\notin L\check R^{-1}_z$.

If $(x,y)\notin L^{-2}L^2$, then $B(y;L^{-2})\cap B(x;L^{-2})=\emptyset$. Since $y\in X=B(Z_L;\tilde L)$, there is $z\in Z_L$ such that $y\in B(z;\tilde L)$. Then $z\in B(y;\tilde  L^{-1})\subset B(y;L^{-1})$ and the $L^{-1}$-ball $B(z;L^{-1})\subset B(y;L^{-2})$ does not intersect $B(x;L^{-2})$, which implies $B(z;L^{-1}L)\cap B(x;L^{-1})=\emptyset$. Observe that $B(y;\tilde R_z^{-1})\subset B(z;\tilde L\tilde R_z^{-1})\subset B(z;R_z^{-1}L)\subset B(z;L^{-1}L)$ and hence $B(y;\tilde R_z^{-1})\cap B(x;L^{-1})\subset B(z;L^{-1}L)\cap B(x;L^{-1})=\emptyset$. So, $(x,y)\notin L^{-1}\tilde R_z$. Since $\check R_z\tilde L^{-1}\subset L^{-1}\tilde R_z$, we get also $(x,y)\notin \check R_z\tilde L^{-1}$.

This completes the proof of the equality $\bigcap_{(L,R)\in\mathcal P}LR^{-1}\cap RL^{-1}=\Delta_X$, which implies the desired inequality $$\psi(\FU)\le|\mathcal P|\le\sum_{L\in\Lambda}|Z_L|\le\psi(\mathcal L\mathcal L^{-1}\vee\mathcal L^{-1}\mathcal L)\cdot\ell(\mathcal L)\cdot\ell(\mathcal L^{-1}).$$
\end{proof}

In Section~\ref{s7} we shall need the following upper bound on the local pseudocharacters  $\dot\psi(\mathcal L\mathcal L^{-1})$ and $\dot\psi(\mathcal R\mathcal R^{-1})$ of normally $\pm$-subcommuting quasi-uniformities $\mathcal L$ and $\mathcal R$.

\begin{proposition}\label{p5.1} If the topology of a Hausdorff space $X$ is generated by two normally $\pm$-subcommuting quasi-uniformities $\mathcal L$ and $\mathcal R$, then  $\dot\psi(\mathcal L\mathcal L^{-1})\le\overline{\psi}(X)\cdot \ell^{\pm 2}(X)$ and $\dot\psi(\mathcal R\mathcal R^{-1})\le\overline{\psi}(X)\cdot \ell^{\pm 2}(X)$.
\end{proposition}

\begin{proof} First we prove that $\dot\psi(\mathcal L\mathcal L^{-1})\le\overline{\psi}(X)\cdot \ell^{\pm 2}(X)$. Fix any point $x\in X$.
Since the topology of $X$ is generated by the quasi-uniformity $\mathcal R$, we can fix a subfamily $\mathcal R_x\subset\mathcal R$ of cardinality  $|\mathcal R_x|\le \overline{\psi}_x(X)\le\overline{\psi}(X)$ such that $\bigcap_{R\in\mathcal R_x}\overline{B(x;RRR)}=\{x\}$.

By the normality of the quasi-uniformity $\mathcal R$, for every $R\in\mathcal R_x$ we get $\overline{B(x;RR)}\subset \overline{B(x;RRR)}^\circ$. Then for every point $z\in X\setminus \overline{B(x;RRR)}^\circ$ we can find an entourage $L_z\in \mathcal L$ such that  $B(z;L_zL_z)\cap \overline{B(x;RR)}=\emptyset$.
For every point $z\in \overline{B(x;RRR)}^\circ$ choose an entourage $L_z\in\mathcal L$ such that $B(z;L_zL_z)\subset \overline{B(x;RRR)}^\circ$. Since the quasi-uniformities $\mathcal L$ and $\mathcal R$ are normally $\pm$-subcommuting, for the entourage $R\in\mathcal R$ there is an entourage $\tilde R\in\mathcal R$ such that for every entourage $L\in\mathcal L$ there is an entourage $\tilde L\in\mathcal L$ such that $\tilde R^{-1}\tilde L\subset LR^{-1}$. In particular, for every $z\in Z$ there is an entourage $\tilde L_z\in\mathcal L$ such that $\tilde R^{-1} \tilde L_z\subset L_zR^{-1}$. Replacing $\tilde L_z$ by a smaller entourage we can assume that $\tilde L_z\subset L_z$ and $B(x;\tilde L_z)\subset B(x;R)$.

By the definition of $\ell^{\pm2}(X)$, for the neighborhood assignment $N_R=\bigcup_{z\in X}\{z\}\times B(z;\tilde L_z\cap \tilde R)$ there is a subset $Z_R\subset X$ of cardinality $|Z_R|\le\ell^{\pm2}(X)$ such that $X=B(Z_R;N_RN_R^{-1})$.

We claim that the subfamily $\mathcal L'=\bigcup_{R\in\mathcal R_x}\{\tilde L_z:z\in Z_R\}\subset\mathcal L$ has the required property: $\bigcap_{L\in\mathcal L'}B(x;LL^{-1})=\{x\}$.
Given any point $y\in X\setminus \{x\}$, find an entourage $R\in\mathcal R_x$ such that $y\notin\overline{B(x;RRR)}$. Since $y\in X=B(Z_R;N_RN_R^{-1})$, there is a point $z\in Z_R$ such that $y\in B(z;N_RN_R^{-1})$ and hence $B(y;L_y\cap\tilde R)\cap B(z;L_z\cap\tilde R)=B(y;N_R)\cap B(z;N_R)\ne\emptyset$ and $y\in B(z;L_z\tilde R^{-1})$.
Since $y\notin\overline{B(x;RRR)}^\circ$, the choice of the entourages $L_y,L_z$ implies that $z\notin \overline{B(x;RRR)}^\circ$. We claim that $B(y;\tilde L_z)\cap B(x;\tilde L_z)=\emptyset$. To derive a contradiction, assume that
$B(y;\tilde L_z)\cap B(x;\tilde L_z)\ne\emptyset$. Then $$\emptyset \ne B(y;\tilde L_z)\cap B(x;\tilde L_z)\subset B(z;\tilde L_z\tilde R^{-1}\tilde L_z)\cap B(x;R)\subset B(z;\tilde L_zL_zR^{-1})\cap B(x;R)$$ and hence $ B(z;L_zL_z)\cap B(x;RR)\ne \emptyset$, which contradicts the choice of the entourage $L_z$. This contradiction completes the proof of the inequality $\dot\psi(\mathcal L\mathcal L^{-1})\le \overline{\psi}(X)\cdot \ell^{\pm 2}(X)$.

By analogy (or changing $\mathcal L$ and $\mathcal R$ by their places)
we can prove that $\dot\psi(\mathcal R\mathcal R^{-1})\le \overline{\psi}(X)\cdot \ell^{\pm 2}(X)$.
\end{proof}

\section{Quasi-uniformities on topological monoids}\label{s6}

A {\em topological monoid} is a topological semigroup $X$ possessing a (necessarily unique) two-sided unit $e\in X$. We shall say that a topological monoid $S$ has {\em open shifts} if for any elements $a,b\in X$ the two-sided shift $s_{a,b}:X\to X$, $s_{a,b}:x\mapsto axb$, is an open map.

A typical example of a topological monoid with open shifts is a {\em paratopological group}, i.e., a group endowed with a topology making the group operation $G\times G\to G$, $(x,y)\mapsto xy$, continuous.

The closed half-line $[0,\infty)$ endowed the Sorgenfrey topology (generated by the base $\mathcal B=\{[a,b):0\le a<b<\infty\}$) and the operation of addition of real numbers is a topological monoid with open shifts, which is not a (paratopological) group.

Each topological monoid $X$ carries five natural quasi-uniformities:
\begin{itemize}
\item the {\em left quasi-uniformity} $\mathcal L$, generated by the base $\big\{\{(x,y)\in X\times X:y\in xU\}:U\in\mathcal N_e\big\}$,
\item the {\em right quasi-uniformity} $\mathcal R$, generated by the base $\big\{\{(x,y)\in X\times X:y\in Ux\}:U\in\mathcal N_e\big\}$,
\item the {\em two-sided quasi-uniformity} $\mathcal L\vee \mathcal R$, generated by the base $\big\{\{(x,y)\in X\times X:y\in Ux\cap xU\}:U\in\mathcal N_e\big\}$,
\item the {\em Roelcke quasi-uniformity} $\mathcal R\mathcal L=\mathcal L\mathcal R$, generated by the base $\big\{\{(x,y)\in X\times X:y\in UxU\}:U\in\mathcal N_e\big\}$, and
\item the {\em quasi-Roelcke uniformity} $\FU=\mathcal R\mathcal L^{-1}\vee\mathcal L\mathcal R^{-1}$, generated by the base\newline $\big\{\{(x,y)\in X\times X:Ux\cap yU\ne\emptyset\ne Uy\cap xU\}:U\in\mathcal N_e\big\}$.
\end{itemize}
Here by $\mathcal N_e$ we denote the family of all open neighborhoods of the unit $e$ in $X$.
The quasi-uniformities $\mathcal L$, $\mathcal R$, $\mathcal L\vee\mathcal R$, and $\mathcal R\mathcal L$ are well-known in the theory of topological and paratopological groups (see \cite[Ch.2]{RD}, \cite[\S1.8]{AT}). The quasi-Roelcke uniformity was recently introduced in \cite{BR}. It should be mentioned that on topological groups the quasi-Roelcke uniformity coincides with the Roelcke (quasi-)uniformity. The following diagram describes the relation between these five quasi-uniformities (an arrow $\U\to\V$ in the diagram indicates that $\U\subset\V$).
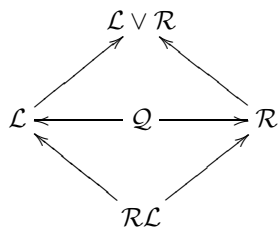
$$\xymatrix{
&\mathcal L\vee\mathcal R\\
\mathcal L\ar[ru]&\FU\ar[r]\ar[l]&\mathcal R\ar[lu]\\
&\mathcal R\mathcal L\ar[lu]\ar[ru]
}$$

If a topological monoid $X$ has open shifts, then the quasi-uniformities $\mathcal L$, $\mathcal R$, $\mathcal L\vee\mathcal R$ and $\mathcal R\mathcal L$ generate the original topology of $X$  (see \cite{Koper}, \cite{KMR}) whereas the quasi-Roelcke uniformity $\FU$ generates a topology $\tau_{\FU}$, which is (in general, strictly) weaker than the topology $\tau$ of $X$. If $X$ is a paratopological group, then the topology $\tau_{\FU}$ on $G$ coincides with the joint $\tau_2\vee (\tau^{-1})_2$ of the second oscillator topologies considered by the authors in \cite{BR02}. The topology $\tau_{\FU}$ turns the paratopological group into a quasi-topological group, i.e., a group endowed with a topology in which the inversion and all shifts are continuous (see Proposition~\ref{p7.3}).

\begin{proposition}\label{p6.1} On each topological monoid $X$ with open shifts the quasi-uniformities $\mathcal L$ and $\mathcal R$ are normally commuting, normally $\pm$-subcommuting, and normal. The topology of $X$ is Hausdorff if and only if the quasi-Roelcke uniformity $\FU=\mathcal L\mathcal R^{-1}\vee\mathcal R\mathcal L^{-1}$ on $X$ is separated.
\end{proposition}

\begin{proof} To see that the quasi-uniformities $\mathcal L$ and $\mathcal R$ are normally commuting and normally $\pm$-subcommuting, fix any entourage $L\in\mathcal L$ and find a neighborhood $U\subset G$ of the unit $e$ such that $\tilde L=\{(x,y)\in X\times X:y\in xU\}\subset L$. Given any entourage $R\in\mathcal R$, find a neighborhood $V\subset G$ of the unit $e$ such that $\tilde R=\{(x,y)\in X\times X:y\in Vx\}\subset R$. Then
$$
\begin{aligned}
\tilde L\tilde R&=\{(x,y)\in X\times X:\mbox{$\exists z\in X$ such that $(x,z)\in \tilde L$ and $(z,y)\in\tilde R$}\}=\\
&=\{(x,y)\in X\times X:\mbox{$\exists z\in X$ such that $z\in xU$ and $y\in Vz$}\}=\\
&=\{(x,y)\in X\times X:y\in V(xU)\}=\{(x,y)\in X\times X:y\in (Vx)U\}=\tilde R\tilde L\subset RL\cap LR.
\end{aligned}
$$
This implies that the quasi-uniformities $\mathcal L$ and $\mathcal R$ are normally commuting.

Next, we prove that $\tilde L^{-1}\tilde R\subset \tilde R\tilde L^{-1}\subset RL^{-1}$. Given any pair $(x,y)\in\tilde L^{-1}\tilde R$, find a point $z\in X$ such that $(x,z)\in\tilde L^{-1}$ and $(z,y)\in\tilde R$. Then $x\in zU$ and $y\in Vz$. So, we can find points $u\in U$ and $v\in V$ such that $x=zu$ and $y=vz$. Multiplying $x=zu$ by $v$, we get $vx=vzu=yu$ and hence $(x,vx)\in \tilde R$ and $(y,vx)=(y,yu)\in \tilde L$, which implies that $(x,y)\in\tilde R\tilde L^{-1}\subset RL^{-1}$. So, $\tilde L^{-1}\tilde R\subset \tilde R\tilde L^{-1}\subset RL^{-1}$. By analogy we can prove that $\tilde R^{-1}\tilde L\subset \tilde L\tilde R^{-1}\subset LR^{-1}$.

By Proposition~\ref{p4.3}, the quasi-uniformities $\mathcal L$ and $\mathcal R$, being normally $\pm$-subcommuting, are normal.
\smallskip

If $X$ is Hausdorff, then for any distinct points $x,y\in X$ we can find a neighborhood $U\subset X$ of the unit $e$ such that $Ux\cap yU=\emptyset$. Then for the entourages $L=\{(x,y)\in X:y\in xU\}\in\mathcal L$ and $R=\{(x,y)\in X\times X:y\in Ux\}$ we get $y\notin B(x;RL^{-1})\supset B(x;RL^{-1}\cap LR^{-1})$. This means that $\bigcap\FU=\Delta_X$ and the quasi-Roelcke uniformity  $\FU$ is separated.

Now assume that the quasi-Roelcke uniformity $\FU$ is separated. Given two distinct points $x,y\in X$, find two entourages $L\in\mathcal L$ and $R\in\mathcal R$ such that $(x,y)\notin LR^{-1}\cap RL^{-1}$ and hence $(x,y)\notin LR^{-1}$ or $(x,y)\notin RL^{-1}$. For the entourages $L,R$, find a neighborhood $U\subset X$ of $e$ such that
$\{(x,y)\in X\times X:y\in xU\}\subset L$ and $\{(x,y)\in X\times X:y\in Ux\}\subset R$. If $(x,y)\notin LR^{-1}$, then $xU\cap Uy=\emptyset$. If $(x,y)\in RL^{-1}$, then $Ux\cap yU=\emptyset$. In both cases the points $x,y$ has disjoint neighborhoods in $X$, which means that $X$ is Hausdorff.
\end{proof}

Proposition~\ref{p6.1} and Theorem~\ref{t3.3} imply:

\begin{theorem}\label{t6.2} Each Hausdorff topological monoid $X$ with open shifts is functionally Hausdorff and has submetrizability number $sm(X)\le\psi(\FU)\le\chi(X)$ and $i$-weight $iw(X)\le \psi(\FU)\cdot\log(\ell(\FU))\le \chi(X)\cdot \log(\dc(X))$.
\end{theorem}

Observe that for a paratopological group $G$ the quasi-Roelcke uniformity $\FU$ generates the topology of $G$ if and only if $G$ is a topological group.

\begin{problem} Study properties of topological monoids $S$ with open shifts whose topology is generated by the quasi-Roelcke uniformity $\FU$.
\end{problem}

\section{The submetrizability number and $i$-weight of paratopological groups}\label{s7}

In this section we apply the results of the preceding sections to paratopological groups, i.e., groups $G$ endowed with a topology making the group operation $G\times G\to G$, $(x,y)\mapsto xy$, continuous. It is easy to see that the inversion map $G\to G$, $x\mapsto x^{-1}$, is a uniform homeomorphism of the quasi-uniform spaces $(G,\mathcal L^{-1})$ and $(G,\mathcal R)$ and also a uniform homeomorphism of the quasi-uniform spaces $(G,\mathcal R^{-1})$ and $(G,\mathcal L)$. This observation combined with Propositions~\ref{p3.5} and \ref{p6.1} implies:

\begin{proposition}\label{p7.1} On each paratopological group $G$
\begin{enumerate}
\item the quasi-uniformities $\mathcal L$ and $\mathcal R$ are normally commuting, normally $\pm$-subcommuting, and normal;
\item the quasi-uniformities $\mathcal L^{-1}$ and $\mathcal R^{-1}$ are normally commuting, normally $\pm$-subcommuting, and generate the same topology on $G$.
\end{enumerate} If the topology of $G$ is Hausdorff, then the quasi-uniformities $\mathcal L$ and $\mathcal R$ are 3-separated and  the quasi-Roelcke uniformity $\FU=\mathcal L\mathcal R^{-1}\vee\mathcal R\mathcal L^{-1}$ is separated.
\end{proposition}

Next, we prove that a paratopological group endowed with the quasi-Roelcke uniformity is a uniform quasi-topological group.

\begin{definition} A {\em uniform quasi-topological group} is a group $G$ endowed with a uniformity $\U$ such that the inversion $G\to G$, $x\mapsto x^{-1}$, is uniformly continuous and for every $a,b\in G$ the shifts $s_{a,b}:G\to G$, $s_{a,b}:x\mapsto axb$, is uniformly continuous.
\end{definition}

\begin{proposition}\label{p7.3} Any paratopological group $G$ endowed with the quasi-Roelcke uniformity $\FU=\mathcal L\mathcal R^{-1}\vee\mathcal R\mathcal L^{-1}$ is a uniform quasi-topological group.
\end{proposition}

\begin{proof} Observe that for any neighborhood $V\in\mathcal N_e$ and points $x,y\in G$ the inclusion $y\in VxV^{-1}\cap V^{-1}xV$ is equivalent to $y^{-1}\in Vx^{-1}V^{-1}\cap  V^{-1}x^{-1}V$, which implies that the inversion map $G\to G$, $x\mapsto x^{-1}$, is uniformly continuous.

Next, we show that for every $a,b\in G$ the shift $s_{a,b}:G\to G$, $s_{a,b}:x\mapsto axb$, is uniformly continuous. Fix any neighborhood $V\in\mathcal N_e$ of $e$. By the continuity of the shifts on $G$, there exists a neighborhood $U\subset V$ of $e$ such that $aU\subset Va$, $Ub\subset bV$, $Ua^{-1}\subset a^{-1}V$, and $b^{-1}U\subset Vb^{-1}$. Inverting the two latter inclusions, we get $aU^{-1}\subset V^{-1}a$ and $U^{-1}b\subset bV^{-1}$. Then for any points $x,y\in G$ with $y\in U^{-1}xU\cap UxU^{-1}$, we get
$ayb\in aU^{-1}xUb\cap aUxU^{-1}b\subset V^{-1}axbV\cap VaxbV^{-1}$, which  means that the shift $s_{a,b}$ is uniformly continuous.
\end{proof}

The following theorem is a partial case of Theorem~\ref{t6.2}.

\begin{theorem}\label{t7.4} Each Hausdorff paratopological group $G$ is functionally Hausdorff and has submetrizability number $sm(G)\le\psi(\FU)\le\chi(G)$ and $i$-weight $iw(G)\le \psi(\FU)\cdot\log(\ell(\FU))\le \chi(G)\cdot \log(\dc(G))$.
\end{theorem}

In light of this theorem it is important to have upper bound on the pseudocharacter $\psi(\FU)$ of the quasi-Roelcke uniformity. Such upper bounds are given in the following theorem, which unifies or generalizes the results of \cite{IS} and \cite{LL}.

\begin{theorem}\label{t7.5} For any Hausdorff paratopological group $G$ its quasi-Roelcke uniformity $\FU=\mathcal L\mathcal R^{-1}\vee\mathcal R\mathcal L^{-1}$ has pseudocharacter
\begin{enumerate}
\item $\psi(\FU)\le \min\{\psi(\mathcal L\mathcal L^{-1})\cdot\ell(\mathcal L^{-1}),\psi(\mathcal L^{-1}\mathcal L)\cdot\ell(\mathcal L)\}\le
    \overline{\psi}(G)\cdot\ell^{\pm2}(G)\cdot \min\{\ell(\mathcal L),\ell(\mathcal L^{-1})\}\le\newline
    \overline{\psi}(G)\cdot \ell^{\pm2}(G)\cdot\min\{q\ell^{\pm1}(G), q\ell^{\mp1}(G)\}$;
\item $\psi(\FU)\le\psi(\mathcal L\mathcal L^{-1}\vee \mathcal L^{-1}\mathcal L)\cdot\ell(\mathcal L^{-1})\cdot\ell(\mathcal L)\le \psi(\mathcal L^{\vee 2})\cdot q\ell^{\mp1}(G)\cdot q\ell^{\pm1}(G)$;
\item $\psi(\FU)\le \psi(\mathcal L\mathcal L^{-1}\mathcal L)\cdot \ell(\mathcal L\mathcal L^{-1}\vee \mathcal L^{-1}\mathcal L)\le\psi(\mathcal L^{\pm3})\cdot q\ell^{\vee 2}(G)$.
\end{enumerate}
Moreover, if the quasi-uniformity $\mathcal L$ is
\begin{enumerate}
\item[(4)] $\mp4$-separated, then $\psi(\FU)\le\psi(\mathcal L^{-1}\mathcal L\mathcal L^{-1}\mathcal L)\cdot\ell(\mathcal L^{-1}\mathcal L)\cdot \ell^{\pm2}(X)\le \psi(\mathcal L^{\mp4})\cdot q\ell^{\mp2}(X)\cdot \ell^{\pm2}(G)$;
\item[(5)] $\pm6$-separated, then
$\psi(\FU)\le\overline{\psi}(\mathcal L\mathcal L^{-1}\mathcal L\mathcal L^{-1}\mathcal L)\cdot \ell^{\pm2}(G)=\overline{\psi}(\mathcal L^{\pm5})\cdot \ell^{\pm2}(G)$.
\end{enumerate}
\end{theorem}

\begin{proof} 1. The inequality $\psi(\FU)\le\psi(\mathcal L\mathcal L^{-1})\cdot\ell(\mathcal L^{-1})$ follows from Theorem~\ref{t4.5}(2), which also implies
$\psi(\FU)\le\psi(\mathcal R\mathcal R^{-1})\cdot\ell(\mathcal R^{-1})=\psi(\mathcal L^{-1}\mathcal L)\cdot\ell(\mathcal L)$. By Proposition~\ref{p5.1}, $\psi(\mathcal L\mathcal L^{-1})=\dot\psi(\mathcal L\mathcal L^{-1})\le\overline{\psi}(G)\cdot\ell^{\pm2}(G)$ and $\psi(\mathcal L^{-1}\mathcal L)=\psi(\mathcal R\mathcal R^{-1})=\dot\psi(\mathcal R\mathcal R^{-1})\le\overline{\psi}(G)\cdot\ell^{\pm2}(G)$, which implies
$$\min\{\psi(\mathcal L\mathcal L^{-1})\cdot\ell(\mathcal L^{-1}),\psi(\mathcal L^{-1}\mathcal L)\cdot\ell(\mathcal L)\}\le\overline{\psi}(G)\cdot\ell^{\pm2}(G)\cdot \min\{\ell(\mathcal L),\ell(\mathcal L^{-1})\}.$$
\smallskip

2, 3. The upper bounds from the second and third items follow from Theorem~\ref{t4.5}(9,7) and Proposition~\ref{p7.1}.
\smallskip

4. Assume that the quasi-uniformity $\mathcal L$ is $\mp 4$-separated. Then we can choose a subfamily $\U\subset \mathcal N_e$ of cardinality $|\U|=\psi(\mathcal L^{-1}\mathcal L\mathcal L^{-1}\mathcal L)$ such that $\bigcap_{U\in\U}U^{-1}UU^{-1}U=\{e\}$. Replacing every $U$ by a smaller neighborhood of $e$, we can assume that $\bigcap_{U\in\U}U^{-2}UU^{-1}U=\{e\}$. Since $\overline{U^{-1}UU^{-1}U}\subset U^{-1}(U^{-1}UU^{-1}U)$, we conclude that $\bigcap_{U\in\U}\overline{U^{-1}UU^{-1}U}=\{e\}$ and $\overline{\psi}(\mathcal L^{-1}\mathcal L\mathcal L^{-1}\mathcal L)\le|\U|=\psi(\mathcal L^{-1}\mathcal L\mathcal L^{-1}\mathcal L)$. Applying Theorem~\ref{t4.4}(4) to the pre-uniformity $\A=\mathcal L^{-1}\mathcal L$, we get the upper bound
$$\psi(\FU)\le\overline{\psi}(\A^{-1}\A\U)\cdot \ell(\A)\cdot\ell^{\pm2}(G)=\overline{\psi}(\mathcal L^{-1}\mathcal L\mathcal L^{-1}\mathcal L\mathcal L)\cdot\ell(\mathcal L^{-1}\mathcal L)\cdot\ell^{\pm2}(G)=\psi(\mathcal L^{-1}\mathcal L\mathcal L^{-1}\mathcal L)\cdot \ell(\mathcal L^{-1}\mathcal L)\cdot\ell^{\pm2}(G).$$

5. The fifth item follows from Theorem~\ref{t4.5}(6).
\end{proof}

\section{Two counterexamples}\label{s8}

In this section we construct two examples of paratopological groups that have some rather unexpected properties.

\subsection{A paratopological group with countable pseudocharacter which is not submetrizable}
In Theorem~\ref{t7.5}(1) we proved that for each Hausdorff paratopological group $G$ its quasi-Roelcke uniformity has pseudocharacter $\psi(\FU)\le\overline{\psi}(G)\cdot\ell^{\pm2}(G)\cdot \min\{\ell(\mathcal L),\ell(\mathcal L^{-1}\}$. It is natural to ask if this upper bound can be improved to $\psi(\FU)\le\overline{\psi}(G)$. In this section we show that this inequality is not true in general.
Namely, we present an example of a zero-dimensional (and hence) Hausdorff abelian paratopological group which has countable pseudocharacter but is not submetrizable. Some properties of this group can be proved only under Martin Axiom \cite{Weiss}, whose topological equivalent says that each countably cellular compact Hausdorff space is $\kappa$-Baire for every cardinal $\kappa<\mathfrak c$. We say that a topological space $X$ is {\em $\kappa$-Baire} if for any family $\U$ consisting of $\kappa$ many open dense subsets of $X$ the intersection $\bigcap\U$ is dense in $X$. Under Martin's Axiom for $\sigma$-centered posets, each separable compact Hausdorff space is $\kappa$-Baire for every cardinal $\kappa<\mathfrak c$. This implies that under Martin's Axiom (for $\sigma$-centered posets) the space $\IZ^\kappa$ endowed with the Tychonoff product topology is $\kappa$-Baire for every cardinal $\kappa<\mathfrak c$. Here $\mathfrak c$ stands for the cardinality of continuum. In the statement (4) of the following theorem by $\mathfrak d$ we denote the cofinality the partially ordered set $(\IN^\w,\le)$.
It is known \cite{Vaugh} that $\w_1\le\mathfrak d\le\mathfrak c$ and $\mathfrak d=\mathfrak c$ under Martin's Axiom (for countable posets).

Let $\kappa$ be an uncountable cardinal. On the group $\IZ^\kappa$ of all functions $g:\kappa\to\IZ$ consider the shift-invariant topology $\tau_\uparrow$ whose neighborhood base at the zero function $e:\kappa\to\IZ$ consists of the sets
$$W_{F,m}=\big\{g\in \IZ^\kappa:g|F=0,\;g(\kappa)\subset \{0\}\cup [m,\infty)\big\}$$where $m\in\IN$ and $F$ runs over finite subsets of $\kappa$. The group $\IZ^\kappa$ endowed with the topology $\tau_\uparrow$ is a paratopological group, denoted by ${\uparrow}\IZ^\kappa$.  Since the group ${\uparrow}\IZ^\kappa$ is abelian, the fours standard uniformities of ${\uparrow}\IZ^\kappa$ coincide (i.e., $\mathcal L=\mathcal R=\mathcal L\vee\mathcal R=\mathcal R\mathcal L$) whereas the quasi-Roelcke uniformity $\FU$ coincides with the pre-uniformities $\mathcal L\mathcal L^{-1}$ and $\mathcal R\mathcal R^{-1}$.

\begin{theorem}\label{t8.1} For any uncountable cardinal $\kappa$ the paratopological group $G={\uparrow}\IZ^\kappa$ has the following properties:
\begin{enumerate}
\item $G$ is a zero-dimensional (and hence regular) Hausdorff abelian paratopological group;
\item the topology on $G$ induced by the quasi-Roelcke uniformity $\FU$ coincides with the Tychonoff product topology $\tau$ on $\IZ^\kappa$;
\item $\psi(\FU)=\chi(G)=\kappa$ but $\psi(G)=\overline{\psi}(G)=\w$;
\item $\ell(\FU)=\w$ but $\ell(\mathcal L)\ge\mathfrak d>\w$;
\item $c(G)\ge\kappa$ but $\dc(G)=\w$;
\item $iw(G)\cdot\w=sm(G)\cdot \w\ge\log(2^\kappa)$.
\item If $2^\kappa>\mathfrak c$, then $G$ is not submetrizable.
\item If the space $\IZ^\kappa$ is $\kappa$-Baire, then $G$ fails to have $G_\delta$-diagonal and hence is not submetrizable.
    \end{enumerate}
\end{theorem}

\begin{proof} 1. It is clear that the topology $\tau_\uparrow$ on ${\uparrow}\IZ^\kappa$ is stronger than the Tychonoff product topology $\tau$ on $\IZ^\kappa$. This implies that the paratopological group $G={\uparrow}\IZ^\kappa$ is Hausdorff.
Observing that each basic neighborhood $W_{F,m}$ of the zero function $e\in\IZ^\kappa$ is $\tau$-closed, we conclude that it is $\tau_\uparrow$-closed, which implies that the space  ${\uparrow}\IZ^\kappa$ is  zero-dimensional and hence regular.
\smallskip

2. Observe that for every basic neighborhood $W_{F,m}$ of zero, the set $W_{F,m}-W_{F,m}$ coincides with the basic neighborhood $W_F=\{g\in\IZ^\kappa:g|F=0\}$ of zero in the Tychonoff product topology $\tau$. This implies that $\tau$ coincides with the topology induced by the quasi-Roelcke uniformity $\FU$.
\smallskip

3. The equality $\chi(G)=\kappa=\psi(\FU)$ easily follows from the definition of the topology $\tau_\uparrow$ and the fact that the quasi-Roelcke uniformity $\FU$ generates the Tychonoff product topology on $\IZ^\kappa$. To see that $\psi(G)=\overline{\psi}(G)=\w$, observe that $\bigcap_{m\in\IN}W_{\emptyset,m}=\{e\}$.
\smallskip

4. To see that $\ell(\FU)=\w$, take any basic open neighborhood $W_{F,m}$ of zero in the group $G$ and observe that $\IZ^F=\{g\in \IZ^\kappa:g|\kappa\setminus F=0\}$ is a countable subgroup of $G$ such that $G=\IZ^F+(W_{F,m}-W_{F,m})$, which implies that $\ell(\FU)\le\w$.  On the other hand, the boundedness number $\ell(\mathcal L)$ of the left quasi-uniformity on the paratopological group ${\uparrow}\IZ^\kappa$ is equal to the cofinality of the partially ordered set $(\IN^\kappa,\le)$ which is not smaller that $\mathfrak d$, the cofinality of the partially ordered set $(\IN^\w,\le)$.
\smallskip

5. For every $x\in\kappa$ denote by $\delta_x:\kappa\to\{0,1\}\subset\IZ$ the characteristic function of the singleton $\{x\}$ and let $U_x=\delta_x+W_{\{x\},2}$ be a basic neighborhood of $\delta_x$. We claim that for any distinct points $x,y\in\kappa$ the sets $U_x$ and $U_y$ are disjoint.
To derive a contradiction, assume that $U_x\cap U_y$ contains some function $f\in\IZ^\kappa$. The inclusion $f\in U_x$ implies that $f(x)=\delta_x(x)=1$. On the other hand, $f\in U_y$ implies $f(x)\in\{\delta_y(x)\}\cup[\delta_y(x)+2,\infty)=\{0\}\cup[2,\infty)\not\ni 1$. So, the closed-and-open sets $U_x$, $x\in\kappa$, are pairwise disjoint and hence $c(G)\ge |\{U_x\}_{x\in\kappa}|=\kappa$.

By Proposition~\ref{p1.8}, $\dc(G)=\ell^{\pm4}(G)$. So, it suffices to prove that $\ell^{\pm4}(G)=\w$. Given a neighborhood assignment $V$ on $G$, we need to find a countable subset $C\subset G$ such that $B(C;VV^{-1}VV^{-1})=G$. Using Zorn's Lemma, find a maximal subset $C\subset G$ such that $B(x;VV^{-1})\cap B(y;VV^{-1})=\emptyset$ for any distinct points $x,y\in C$. By the maximality of $C$, for every $x\in G$ there is a point $c\in C$ such that $B(c;VV^{-1})\cap B(x;VV^{-1})\ne\emptyset$, which implies $x\in B(C;VV^{-1}VV^{-1})$ and hence $X=B(C;VV^{-1}VV^{-1})$. It remains to prove that the set $C$ is countable. To derive a contradiction, assume that $C$ is uncountable. For every $x\in G$ find a finite subset $F_x\subset\kappa$ and a positive number $m_x\in\IN$ such that $x+W_{F_x,m_x}\subset B(x;V)$. By the $\Delta$-system Lemma \cite[16.1]{JW2}, the uncountable set $C$ contains an uncountable subset $D\subset C$ such that the family $(F_x)_{x\in D}$ is a $\Delta$-system with kernel $K$, which means that $F_x\cap F_y=K$ for any distinct points $x,y\in D$. For every $n\in\IN$ and $f\in\IZ^K$ consider the subset $D_{n,f}=\{x\in D:x|K=f,\;m_x\le n,\;\sup_{\alpha\in F_x}|x(\alpha)|\le n\}$ of $D$ and observe that $D=\bigcup_{n\in\IN}\bigcup_{f\in\IZ^K}D_{n,f}$. By the Pigeonhole Principle, for some $n\in\IN$ and $f\in\IZ^K$ the set $D_{n,f}$ is uncountable. Consider the clopen subset $\IZ^\kappa(f)=\{x\in\IZ^\kappa:x|K=f\}$ of $\IZ^\kappa$.
Since $\IZ^\kappa(f)$ is a Baire space, for some $m\in\IN$ the set $X_m=\{x\in\IZ^{\kappa}(f):m_x=m\}$ is not nowhere dense in $\IZ^\kappa(f)$.  Consequently, there is a finite subset $\bar K\subset\kappa$ containing $K$ and a
function $\bar f:\bar K\to\IZ$ such that the set $X_m\cap \IZ^\kappa(\bar f)$ is dense in $\IZ^\kappa(\bar f)=\{x\in \IZ^\kappa:x|\bar K=\bar f\}$. Since the family $(F_x\setminus K)_{x\in D}$ is disjoint, the set $\{x\in D:(F_x\setminus K)\cap\bar K\ne\emptyset\}$ is finite, so we can find two functions $x,y\in D_{n,f}$ such that $(F_x\cup F_y)\cap \bar K=K$.
Put $\tilde K=F_x\cup F_y\cup K$ and choose any function $\tilde f:\tilde K\to\IZ$ such that $\tilde f|\bar K=\bar f$ and $f(\alpha)<-n-m$ for any $\alpha\in \tilde K\setminus\bar K$. The function $\tilde f$ determines a non-empty open set $\IZ^\kappa(\tilde f)=\{z\in\IZ^\kappa:z|\tilde K=\tilde f\}$, which contains some function $z\in X_m$ (by the density of $X_m\cap\IZ^\kappa(\bar f)$ in $\IZ^\kappa(\bar f)$). Choose a function $\tilde z\in\IZ^\kappa$ such that $\tilde z|F_x=x|F_x$ and $\tilde z(\alpha)\ge \max\{m+z(\alpha),m_x+x(\alpha)\}$ for every $\alpha\in\kappa\setminus F_x$. Then $\tilde z\in (z+W_{F_z,m})\cap (x+W_{F_x,m_x})\subset B(z;V)\cap B(x;V)$, which implies $z\in B(x;VV^{-1})$. By analogy we can prove that $z\in B(y;VV^{-1})$. So, $B(x;VV^{-1})\cap B(y;VV^{-1})\ne\emptyset$, which contradicts the choice of the set $C\ni x,y$. This contradiction shows that $C$ is countable and hence $\dc(G)=\ell^{\pm4}(G)=\w$.
\smallskip

6. By Proposition~\ref{p1.3}, $iw(G)\cdot\w=sm(G)\cdot\log(\dc(G))=sm(G)\cdot\w$. On the other hand, $2^\kappa=|G|\le |[0,1]^{iw(G)}|=|2^{iw(G)\cdot\w}|$ implies that
$\log(2^\kappa)\le iw(G)\cdot\w$.
\smallskip

7. If $2^\kappa>\mathfrak c$, then $sm(G)\cdot\w\ge\log(2^\kappa)\ge \log(\mathfrak c^+)>\w$, which implies that $sm(G)>\w$ and hence $G$ is not submetrizable.
\smallskip

8. Suppose that the space $\IZ^\kappa$ is $\kappa$-Baire. Assuming that the space $G={\uparrow}\IZ^\kappa$ has $G_\delta$-diagonal, we can apply Theorem 2.2 in \cite{Grue} and find a countable family $(\U_n)_{n\in\IN}$ open covers of $G$, which separates the points of $G$ in the sense that for every distinct points $f,g\in G$ there is $n\in\IN$ such that no set $U\in\U_n$ contains both points $f$ and $g$. Since the space $G$ is zero-dimensional, we can assume that each set $U\in\bigcup_{n\in\w}\U_n$  is closed-and-open in $G$. Put $\U_0=\{G\}$.

We shall construct an increasing sequence $(F_n)_{n\in\w}$ of finite subsets and a sequence $f_n\in \IZ^{F_n}$, $n\in\w$, of functions such that for every $n\in\w$ the clopen set $\IZ^{\kappa}(f_n)=\{f\in\IZ^\kappa:f|F_n=f_n\}$ is contained in $U_n\cap\IZ^\kappa(f_{n-1})$ for some set $U_n\in\U_n$.

We start the inductive construction letting $F_0=\emptyset$ and $f_0:\emptyset\to\IZ$ be the unique function. Then $\IZ^\kappa(f_0)=\IZ^\kappa\in\U_0$. Assume that for some $n\in\IZ$ we have defined a finite set $F_{n-1}\subset\kappa$ and a function $f_{n-1}\in\IZ^{F_{n-1}}$ such that $\IZ^\kappa(f_{n-1})\subset U_{n-1}$ for some $U_{n-1}\in\U_{n-1}$.

The $\F$ be the family of all triples $(F,f,m)$ where $F$ is a finite subset of $\kappa$ containing $F_{n-1}$, $f:F\to\IZ$ is a function extending the function $f_{n-1}$ and $m\in\IN$ is a positive integer. Observe that $|\F|=\kappa$. For every function $g\in {\uparrow}\IZ^{\kappa}$ choose a closed-and-open subset $U_g\in\U_n$ containing $g$ and choose a finite subset $F_g\subset\kappa$ containing $F_{n-1}$ and a number $m_g$ such that $g+W_{F_g,m_g}\subset U_g$. For every triple $(F,f,m)\in\F$ consider the subset $Z_{(F,f,m)}=\{g\in{\uparrow}\IZ^\kappa:(F_g,g|F_g,m_g)=(F,f,m)\}$ and observe that $\IZ^\kappa(f_{n-1})=\bigcup_{(F,f,m)\in\F}Z_{F,f,m}$. Since the space $\IZ^\kappa(f_{n-1})$ is $\kappa$-Baire, there is a triple $(F,f,m)\in\F$ such that the set $Z_{(F,f,m)}$ is not nowhere dense in $\IZ^\kappa(f_{n-1})$. Consequently we can find a finite set $F_n\subset \kappa$ and a function $f_n\in\IZ^{F_n}$ such that for the basic open set $\IZ^\kappa(f_n)=\{g\in \IZ^\kappa:g|F_n=f_n\}$ the intersection $\IZ^\kappa(f_n)\cap Z_{(F,f,m)}$ is dense in $\IZ^\kappa(f_n)$. It follows that $F_n\supset F\supset F_{n-1}$ and $f_n|F=f$. Choose any point $g\in Z_{(F,f,m)}\cap \IZ^\kappa(f_n)$.

We claim that $\IZ^\kappa(f_n)\subset U_g\in\U$. Assuming that $\IZ^\kappa(f_n)\not\subset U_g$, choose a function $h\in \IZ^\kappa(f_n)\setminus U_g$ and find a basic neighborhood $h+W_{E,l}\subset \IZ^{\kappa}(f_n)\setminus U_g$ of $h$. It follows from the inclusion $h+W_{E,l}\subset \IZ^{\kappa}(f_n)$ that $E\supset F_n\supset F$ and $h|F_n=f_n$. Then $h|F=f_n|F=f$. Choose a function $\tilde h:\kappa\to\IZ$ such that $\tilde h|E=h|E$ and $\tilde h(x)\ge \max\{g(x)+m,h(x)+l\}$ for every $x\in\kappa\setminus E$. Then $\tilde h\in (h+W_{E,l})\cap (g+W_{F,m})\subset (\IZ^\kappa(f_n)\setminus U_g)\cap U_g=\emptyset$, which is a desired contradiction completing the inductive step.

After completing the inductive construction, consider the countable set $F_\w=\bigcup_{n\in\w}F_n$ and the function $f_\w:F_\w\to\IZ$ such that $f_\w|F_n=f_n$ for all $n\in\w$. Since the complement $\kappa\setminus F_\w$ is not empty, the ``cube'' $\IZ^\kappa(f_\w)=\{g\in \IZ^\kappa:g|Z_\w=f_\w\}$ contains two distinct functions $f,g$. By the choice of the family $(\U_n)_{n\in\w}$ there is a number $n\in\w$ such that no set $U\in \U_n$ contains both points $f$ and $g$. On the other hand, by the inductive construction, $f,g\in\IZ^{\kappa}(f_\w)\subset \IZ^\kappa(f_n)\subset U_n$ for some set $U_n\in\U$, which is a desired contradiction completing the proof of the theorem.
\end{proof}

\begin{corollary}\label{c8.2} For every cardinal $\kappa\ge\mathfrak c$ the paratopological group ${\uparrow}\IZ^{\mathfrak \kappa}$ has countable pseudocharacter but fails to be submetrizable.
\end{corollary}

It is known \cite{Weiss} that under Martin's Axiom the space $\IZ^{\kappa}$ is $\kappa$-Baire for every cardinal $\kappa<\mathfrak c$. This fact combined with Theorem~\ref{example}(7,8) implies the following MA-improvement of Corollary~\ref{c8.2}.

\begin{corollary} Under Martin's Axiom, for any uncountable cardinal $\kappa$ the paratopological group ${\uparrow}\IZ^{\kappa}$ has countable pseudocharacter but fails to be submetrizable.
\end{corollary}

\begin{problem} Can the space ${\uparrow}\IZ^{\w_1}$ be submetrizable in some model of ZFC?
\end{problem}

In Theorem~\ref{t8.1} we proved that the paratopological group $G={\uparrow}\IZ^\kappa$ has $d(G)\ge c(G)\ge\kappa$ and $\dc(G)=\w$, By Propositions~\ref{p1.2} and \ref{p1.8},  $\ell^{\pm4}(G)=\bar l^{*\kern-1pt1\kern-1pt\frac12}(G)=\dc(G)=\w$. It would be interesting to know the values of some other cardinal characteristics of $G$, intermediate between $\dc(G)$ and $c(G)$.

\begin{problem} For the paratopological group $G={\uparrow}\IZ^\kappa$ calculate the values of cardinal characteristics $\ell^{\pm n}(G)$, $\ell^{\mp n}(G)$, $\ell^{\wedge n}(G)$, $\ell^{\vee n}(G)$ for all $n\in \IN$.
\end{problem}

\subsection{A submetrizable paratopological group whose quasi-Roelcke uniformity has uncountable pseudocharacter}

By Theorem~\ref{t7.4}, each Hausdorff paratopological group $G$ has submetrizability number $sm(G)\le \psi(\FU)$. This inequality can be strict as shown by an example constructed in this subsection.

Given an uncountable cardinal $\kappa$ in the paratopological group ${\uparrow}\IZ^\kappa$ consider the subgroup $H=\{f\in{\uparrow}\IZ^\kappa:|\supp(f)|<\w\}$ consisting of  functions $f:\kappa\to\IZ$ that have finite support $\supp(f)=\{\alpha\in\kappa:f(\alpha)\ne 0\}$. A neighborhood base of $H$ at zero consists of the sets
$$W_{F,m}=\{h\in H:h|F=0,\;h(\kappa)\in\{0\}\cup[m,\infty)\}$$where $F$ runs over finite subsets of $\kappa$ and $m\in\IN$.

\begin{theorem}\label{example} For any uncountable cardinal $\kappa$ the paratopological group $H$ has the following properties:
\begin{enumerate}
\item $H$ is a zero-dimensional (and hence regular) Hausdorff abelian paratopological group;
\item $H$ is strongly $\sigma$-discrete and submetrizable;
\item $iw(H)\cdot\w=\log(\kappa)$;
\item $\psi(\FU)=\chi(H)=\kappa$ but $\psi(H)=\overline{\psi}(H)=\w$;
\item $\ell(\FU)=\w$ but $\ell(\mathcal L)=\dc(H)=\kappa$.
    \end{enumerate}
\end{theorem}

\begin{proof} The items (1), (4), (5) follow (or can be proved by analogy with) the corresponding items of Theorem~\ref{t8.1}.

(2)--(3): To see that the space $H$ is strongly $\sigma$-discrete, write $H$ as $H=\bigcup_{n,m\in\w}H_{n,m}$ where  $H_{n,m}=\{h\in{\uparrow}\IZ^\kappa:|\supp(h)|=n,\;\|h\|\le m\}$ and $\|h\|=\sup_{\alpha\in\kappa}|h(\alpha)|$. We claim that each set $H_{n,m}$ is strongly discrete in $H$. To each function $h\in H_{n,m}$ assign the neighborhood $U_h=h+W_{\supp(h),m+1}$. Given any two distinct functions $g,h\in H_{n,m}$, we shall prove that $U_g\cap U_h=\emptyset$.  Assuming that $U_g\cap U_h$ contains some function $f\in H$, we would conclude that $f|\supp(g)=g|\supp(g)$ and $f|\supp(h)=h|\supp(h)$. So, $g|\supp(g)\cap \supp(h)=h|\supp(g)\cap\supp(h)$ and $g\ne h$ implies that $\supp(g)\ne\supp(h)$. Since $|\supp(g)|=|\supp(h)|=n$, there is $\alpha\in \supp(g)\setminus \supp(h)$ such that $g(\alpha)\ne 0=h(\alpha)$. Then $f(\alpha)\in\{g(\alpha)\}\cap [m+1,\infty)\subset[-m,m]\cap [m+1,\infty)=\emptyset$, which is a contradiction showing that the indexed family $(U_h)_{h\in H_{n,m}}$ is disjoint.

To show that this family $(U_h)_{h\in H_{n,m}}$ is discrete, for every function $g\in H\setminus \bigcup_{h\in H_{n,m}}U_h$ consider its neighborhood $U_g=g+W_{\supp(g),m+1}$. We claim that $U_g\cap U_h=\emptyset$ for every $h\in H_{n,m}$. Assume conversely that for some $h\in H_{n,m}$ the intersection $U_g\cap U_h$ contains a function $f\in H$. Then $f|\supp(g)=g|\supp(g)$ and $f|\supp(h)=h|\supp(h)$, which implies $\supp(g)\ne \supp(h)$. If $\supp(h)\setminus\supp(g)\ne\emptyset$, then we can find $\alpha\in\supp(h)\setminus\supp(g)$ and conclude that $f(\alpha)=h(\alpha)\ne 0=g(\alpha)$ and hence $f(\alpha)\in\{h(\alpha)\}\in [-m,m]\cap [m+1,\infty)=\emptyset$, which is a contradiction. So, $\supp(h)\subset \supp(g)$ and $g|\supp(h)=h|\supp(h)$. It follows from $g\notin U_h$ that for some $\alpha\in\kappa\setminus\supp(h)$ we get $g(\alpha)\notin \{0\}\cup[m+1,\infty)$. Then $\alpha\in\supp(g)$ and $f(\alpha)=g(\alpha)\notin [m+1,\infty)$. On the other hand, the inclusion $f\in U_h$ and $f(\alpha)\ne 0=h(\alpha)$ implies $f(\alpha)\in [m+1,\infty)$. This contradiction completes the proof of the equality $U_g\cap U_h=\emptyset$, which shows that the family $(U_h)_{h\in H_n}$ is discrete in $H$ and the set $H_{n,m}$ is strongly discrete in $H$.
Then the space $H=\bigcup_{n,m\in\w}H_{n,m}$ is strongly $\sigma$-discrete. By Proposition~\ref{p1.1} it is submetrizable and has $i$-weight $iw(H)\cdot\w=\log(|H|)=\log(\kappa)$.
\end{proof}

\end{document}